\documentclass[sn-mathphys,Numbered]{sn-jnl}

\usepackage{graphicx,xcolor}%
\usepackage{multirow}%
\usepackage{amsmath,amssymb,amsfonts, upgreek}%
\usepackage{amsthm,nicefrac}%
\usepackage{mathrsfs}%
\usepackage[title]{appendix}%
\usepackage{xcolor}%
\usepackage{textcomp}%
\usepackage{manyfoot}%
\usepackage{booktabs}%
\usepackage{algorithm}%
\usepackage{algorithmicx}%
\usepackage{algpseudocode}%
\usepackage{listings}%

\usepackage{bm}
\usepackage{mathtools}   
\usepackage{accents}
\usepackage{graphicx} 
\usepackage{comment}
\usepackage{xcolor}
\usepackage{color}
\usepackage{multirow}
\usepackage{hyperref}
\usepackage{cleveref}

\newcommand{\QFOM}{{$\Q_r$-$\mathrm{IRGNM}$}}
\newcommand{\FOM}{{$\mathrm{FOM}$-$\mathrm{IRGNM}$}}
\newcommand{\QTRRB}{{$\Q_r$-$V_r$-$\mathrm{IRGNM}$}}
\newcommand{\TRradius}{\eta}

\newcommand{\Qcoeff}{\mathsf{q}}

\DeclareMathOperator*{\argmin}{arg\,min}

\newcommand{\R}{\mathbb R}
\newcommand{\N}{\mathbb N}
\makeatother
\def\theequation{\thesection.\arabic{equation}}
\newcommand{\bx}{{\bm x}}

\newcommand{\yd}{{y^{\delta}}}

\newcommand{\ubar}[1]{\underaccent{\bar}{#1}}

\newcommand{\qtrial}{{q_{\mbox{\tiny{trial}}}}}

\newcommand{\qa}{{q_{\mathsf a}}}

\newcommand{\qagc}{q_{\mbox{\tiny{AGC}}}}
\renewcommand{\H}{{\mathscr H}}
\newcommand{\Q}{{\mathscr Q}}
\newcommand{\Qad}{{\mathscr Q_\mathsf{ad}}}

\newcommand{\CQ}{{\mathscr C}}
\newcommand{\Sol}{\mathcal S}
\newcommand{\C}{{\mathcal C}}
\newcommand{\F}{{\mathcal F}}
\DeclareMathOperator*{\essinf}{ess\,inf}
\newcommand{\ccont}{ c_{\mbox{\tiny cont}}}
\newcommand{\coercivityconstant}{\ubar{a}}

\newcommand{\ulin}{du}
\newcommand{\plin}{dp}

\definecolor{DarkRed}{rgb}{0.5 0 0}
\definecolor{DarkBlue}{rgb}{0 0 0.5}
\definecolor{DarkGreen}{rgb}{0 0.5 0}
\definecolor{DarkOrange}{rgb}{0.83 0.33 0}
\definecolor{DarkMagenta}{rgb}{0.83 0 0.67}

\newcommand{\ReA}[1]{{#1}}
\newcommand{\ReB}[1]{{#1}}

\newtheoremstyle{Stil}{}{}{\itshape}{}{\bfseries}{.\\}{1ex}{}
\newtheoremstyle{Stil2}{}{}{\rmfamily}{}{\bfseries}{.}{1ex}{}
\theoremstyle{Stil}
\newtheorem{theorem}{Theorem}[section]
\newtheorem{proposition}[theorem]{Proposition}

\theoremstyle{plain}
\newtheorem{remark}[theorem]{Remark}
\newtheorem{assumption}[theorem]{Assumption}

\theoremstyle{Stil2}
\newtheorem{example}[theorem]{Example}

\crefname{assumption}{\textup{Assumption}}{\textup{Assumptions}}
\crefname{lemma}{\textup{Lemma}}{\textup{Lemmas}}
\crefname{theorem}{\textup{Theorem}}{\textup{Theorems}}
\crefname{remark}{\textup{Remark}}{\textup{Remarks}}
\crefname{example}{\textup{Example}}{\textup{Examples}}
\crefname{corollary}{\textup{Corollary}}{\textup{Corollaries}}
\crefname{subsection}{\textup{Section}}{\textup{Subsections}}
\crefname{section}{\textup{Section}}{\textup{Sections}}
\crefname{figure}{\textup{Figure}}{\textup{Figures}}
\crefname{table}{\textup{Table}}{\textup{Tables}}

\begin{document}

\title[Article Title]{Adaptive Reduced Basis Trust Region Methods for Parameter Identification Problems}


\author*[1]{\fnm{Michael} \sur{Kartmann}}\email{michael.kartmann@uni-konstanz.de}

\author[2]{\fnm{Tim} \sur{Keil}}\email{tim.keil@uni-muenster.de}

\author[2]{\fnm{Mario} \sur{Ohlberger}}\email{mario.ohlberger@uni-muenster.de}%

\author[1]{\fnm{Stefan} \sur{Volkwein}}\email{stefan.volkwein@uni-konstanz.de}

\author[3]{\fnm{Barbara} \sur{Kaltenbacher}}\email{barbara.kaltenbacher@aau.at}

\affil*[1]{\orgdiv{Department of Mathematics and Statistics}, \orgname{University of Konstanz}, \orgaddress{\street{Universit\"atsstr. 10}, \city{Konstanz}, \postcode{78464},  \country{Germany}}}

\affil[2]{\orgdiv{Applied Mathematics: Institute for Analysis and Numerics}, \orgname{University of M\"unster}, \orgaddress{\street{Einsteinstr. 62}, \city{M\"unster}, \postcode{48149},  \country{Germany}}}

\affil[3]{\orgdiv{Department of Mathematics}, \orgname{Alpen-Adria-Universit\"at Klagenfurt}, \orgaddress{\street{Universit\"atsstr. 65-67}, \city{Klagenfurt}, \postcode{9020},  \country{Austria}}}


\abstract{
	In this contribution, we are concerned with model order reduction in the context of iterative regularization methods for the solution of inverse problems arising from parameter identification in elliptic partial differential equations. Such methods typically require a large number of forward solutions, which makes the use of the reduced basis method attractive to reduce computational complexity. 
	However, the considered inverse problems are typically ill-posed due to their infinite-dimensional parameter space. Moreover, the infinite-dimensional parameter space makes it impossible to build and certify classical reduced-order models efficiently in a so-called "offline phase". We thus propose a new algorithm that adaptively builds a reduced parameter space in the online phase. The enrichment of the reduced parameter space is naturally inherited from the Tikhonov regularization within an iteratively regularized Gau{\ss}-Newton method.
	Finally, the adaptive parameter space reduction is combined with a certified reduced basis state space reduction within an adaptive error-aware trust region framework. Numerical experiments are presented to show the efficiency of the combined parameter and state space reduction for
	inverse parameter identification problems with distributed reaction or diffusion
	coefficients.
	}

\keywords{parameter identification, reduced basis method, model reduction, inverse problems}



\maketitle

\section{Introduction}\label{section: introduction}

To perform a forward solution of a partial differential equation (PDE), all involved coefficients and parameters must be known. If some parameters cannot be measured directly, they have to be determined from indirect measurements instead. This procedure is called the \emph{inverse problem} or parameter identification in PDEs (cf., e.g., \cite{isakov2017inverse,Kavian+2003+125+162}).
Often these problems are ill-posed in the sense that the reconstructed parameter does not depend continuously on the measurements. To overcome this lack of stability, regularization methods are used to approximate the desired parameter in a stable way \cite{BakKok04,EHNbuch96,KaltenbacherNeubauerScherzer+2008,Kirs96,Voge02}. Regularization can be achieved by different strategies such as Tikhonov, Morozov, or Ivanov regularization or, e.g., by the concept of regularization by discretization. Classical iterative regularization methods are for example the Landweber method, the Levenberg-Marquardt method, or the iteratively regularized Gau{\ss}-Newton method (IRGNM), which we will consider in this paper and which has been extensively studied in different settings \cite{Kaltenbacher_2009, KaltenbacherNeubauerScherzer+2008}. Typically, these methods need a high number of forward solutions of the respective PDE. To improve the efficiency of the algorithm, adaptive finite element approximations were used in \cite{KaltenbacherMain, diss_kirchner}.

Another possibility to reduce the computational complexity is the use of model order reduction (MOR) techniques \cite{bennerohlberger, quarteroni2015reduced}, which have been applied to PDE-constrained optimization \cite{Hinzeetal2009,BennerSachsVolkwein2014,BoncoraglioFarhat2022} and  inverse problems \cite{article, RBL,0061f27858a54ce4bfae8fc2a2015fdc,GhattasWillcox2021}. A particular MOR method is the reduced basis (RB) method \cite{book, Haasdonk2017Chapter2R}, where the full-order model (FOM) is projected onto a low-dimensional space spanned by a reduced basis consisting of snapshots, i.e. of the solution of the PDE corresponding to meaningful parameters. The traditional way of applying the reduced basis method is a so-called offline-online decomposition, where the reduced-order model (ROM) is constructed by a goal-oriented greedy algorithm in the offline stage. In the online stage, the ROM and an a posteriori error estimator are available to perform cheap, but certified computations. The efficiency of the greedy algorithm depends heavily on the parameterization and is restricted to low-dimensional and bounded parameter spaces. Adaptive reduced basis methods for PDE-constrained optimization overcome this behavior partially \cite{YueMeerbergen2013,QGVW17,Keil2021nonco-54293, keil2022adaptive,Banholzer2022-06Trust-57762,APV23} by constructing a small problem-dependent basis along the optimization path. In the context of inverse problems, the authors of \cite{RBL} introduced the RB Landweber algorithm which uses adaptive reduced basis approximations to solve an inverse problem without the theoretical need for the parameter space to be bounded. In their numerical experiments, they tackled parameter spaces of dimension $900$.
Nevertheless, the RB reduction becomes infeasible if the parameter space gets high-dimensional (e.g., as high-dimensional as the underlying FOM approximation) because (i) the projection of all affine components is costly and (ii) the assembly of residual-based error estimators is prohibitively costly since per parameter and per RB element one Riesz-representative needs to be computed, which involves the solution of a system of FOM complexity.

In the literature, there are parameter reduction strategies, such as active subspace methods \cite{Constantin2015}, used by the authors in \cite{Tezzele2018,inproceedings}. In their numerical experiments, the parameter set was of dimension $10$. We also refer to \cite{Lametal2020} for an extension of the active subspace method to the multifidelity case.
In inverse problems, however, the ill-posedness only occurs if the parameter space is infinite-dimensional, i.e. in a discrete setting if the parameter space can be of arbitrary dimension. Correspondingly there is a need for adaptive RB methods that can efficiently deal with a parameter space of arbitrary dimension. 
To this end combined state and parameter reduction methods that are particularly tailored towards the solution of inverse problems are needed. In the context of Bayesian inverse problems such approaches have been proposed e.g. in  \cite{Liebermanetal2016,HimpeOhlberger2014}. We also refer to \cite{HimpeOhlberger2015} for an approach based on cross-gramians in the context of large-scale control problems.

As the main contribution of this paper, we propose an adaptive reduced basis algorithm that can deal with an infinite-dimensional parameter space in an efficient manner and serves as an iterative regularization algorithm for the parameter identification problem. The key ingredient is to reduce the dimension of the parameter space by adaptively constructing a reduced basis for it consisting of the FOM gradients of the discrepancy functional. We show that this is a natural choice that is inherited from the Tikhonov regularization. Since the number of affine components on the reduced parameter space is low-dimensional, the RB method for state space reduction can be applied efficiently. Consequently, our novel approach is a certified, adaptive, error-aware trust region setting to significantly speed up the solution process of the parameter identification problem by parameter- and state space model order reduction.

The article is organized as follows. In Section~\ref{sec:2}, we introduce the parameter identification problem and present the IRGNM. In Section~\ref{section: TRRB}, the adaptive parameter and model reduction approach is described. First, we propose the IRGNM with adaptive parameter space reduction in Section~\ref{sec: QFOM}. On top of that, the classical RB model reduction is applied in Section~\ref{sec: QRB}. In Section~\ref{sec: TRIRGNM}, we present the combined parameter- and state space reduced error-aware trust region iteratively regularized Gau{\ss}-Newton method. In Section~\ref{section: numerical results}, all proposed algorithms are compared in terms of computation time and the quality of the reconstruction, and finally, in Section~\ref{Sec:5} some concluding remarks and an outlook are given.

\section{Parameter identification for elliptic PDEs}
\label{sec:2}
\setcounter{equation}{0}
%
\subsection{Problem formulation}
\label{sec:2.1}

Let $\Q$, $\H$, $V$ be (real) Hilbert spaces and $\Qad\subset\Q$ be a closed, convex set, where the subscript `$\mathsf{ad}$' stands for admissible parameters. We are interested in identifying space-dependent parameters $q\in \Qad$ from noisy measurements $\yd\in\H$ of the exact solution $u\in V$ solving the linear, elliptic partial differential equation (PDE)
\begin{equation}
	\label{eq: VarState}
	a(u,v;q) =\ell(v) \quad \text{for all }v \in V
\end{equation}
with a $q$-dependent bilinear form $a(\cdot\,,\cdot\,;q):V\times V\to \R$ and linear form~$\ell\in V'$. Furthermore, we assume that the exact measurements $y \in \H$ are obtained from the solution $u \in V$ through a linear bounded (observation) operator $\C:V\to\H$, i.e. $y = \C u$.

\begin{assumption}
	\label{A1}
	The bilinear form $a(\cdot\,,\cdot\,;q)$ is assumed to be continuous and coercive for any $q\in\Qad$. Further, for any $u,v\in V$ the  map $q\mapsto a(u,v;q)$ is linear.
\end{assumption}

\begin{remark}
	From the Lax-Milgram theorem (cf., e.g., \emph{\cite[pp.~315-316]{Eva10}}) it is well-known that Assumption~{\em\ref{A1}} implies the existence of a unique solution $u=u(q)\in V$ to the forward problem \eqref{eq: VarState} for any $q\in\Qad$. We can introduce the solution operator $\Sol:\Qad\to V$, where $u(q)=\Sol(q)$. Notice that \eqref{eq: VarState} is a linear problem for any $q\in\Qad$, but $\Sol$ is usually a non-linear operator.
\end{remark}

Next, we express the parameter identification problem as an operator equation by introducing the forward operator 
\begin{align*}
	\F:\Qad\to\H\quad\text{with }\F=\mathcal C\circ\mathcal S.
\end{align*}
The parameter identification problem reads as follows: find a parameter $q^\mathsf e\in\Qad$ such that for a given exact data $y^\mathsf e=\C u^\mathsf e\in \text{range}(\F)$ with $u^\mathsf e=\Sol(q^\mathsf e)$ the following equation holds:
\begin{align}
	\label{eq: IP}
	\tag{$\mathbf{IP}$}
	\F(q^\mathsf e)=y^\mathsf e\quad\text{in }\H.
\end{align}
Throughout our work we assume that the given noisy data $\yd\in \H$ and $y^\mathsf e\in \text{range}(\F)$ satisfy
\begin{align*}
	{\|y^\mathsf e-\yd\|}_\H\leq \delta,
\end{align*}
where the noise level $\delta>0$ is assumed to be known.

\begin{assumption}[Assumptions on the forward operator]
	\label{Ass: forward operator}
	We assume that $\F:\Qad\to\H$ is injective, continuously Fr\'echet-differentiable and that \eqref{eq: IP} is ill-posed in the sense that $\F$ is not continuously invertible.
\end{assumption}
Due to the injectivity in Assumption \ref{Ass: forward operator} and $y^\mathsf e\in \text{range}(\F)$, $q^\mathsf e$ is the unique solution to \eqref{eq: IP}. Further, the ill-posedness of \eqref{eq: IP} poses challenges, since instead of the exact data, only the noisy measurement $\yd$ is available. Therefore, simple inversion fails to provide a reasonable reconstruction of the exact state. Instead one has to apply regularization methods that construct continuous approximations of the discontinuous inverse of $\F$, in order to obtain a stable approximate solution of \eqref{eq: IP}.

Let us introduce two guiding examples that are studied in our numerical experiments carried out in Section~\ref{section: numerical results}.
\begin{example}
	\label{ex: example_problems}
	We consider two possible scenarios based on the examples presented in \cite{Kaltenbacher_2009}. See also the works \cite{HaNeSc95,IK90,KS92,Vol02}. Let $\Omega \subset \R^d$ for $d\in\{1,2,3\}$ be a bounded, convex domain. Using \cite[Corollary 1.2.2.3]{Gri11} this implies that $\Omega$ has a Lipschitz-continuous boundary $\Gamma=\partial\Omega$. We choose $\H=L^2(\Omega)$, $H=L^2(\Omega)$, $V=H^1_0(\Omega)$ and $\C$ to be the compact, injective embedding from $V$ into $\H$.
	\begin{itemize}
		\item[(i)] \textit{Reconstruction of the reaction coefficient}: Let $\qa\in L^\infty(\Omega)$ be a given lower bound with $\qa\ge0$ a.e. (almost everywhere) in $\Omega$. We set $\Q=L^2(\Omega)$ and
		\begin{align}
			\label{Qad}
			\Qad\coloneqq\big\{q\in\Q\,\big|\,\qa\le q\text{ in }\Omega\text{ a.e.}\big\}.
		\end{align}
		Note that $\Qad$ is closed and convex in $\Q$. For given $q\in\Qad$ and $f\in H$ we consider the elliptic problem
		\begin{align}
			\label{PDE1}
		-\Delta u(\bx)+q(\bx)u(\bx)=f(\bx)\text{ f.a.a. }\bx\in\Omega,\quad u(\bx)=0\text{ f.a.a. }\bx\in\Gamma,
		\end{align}
		where `f.a.a.' means `for almost all'. Setting
		\begin{align*}
			a(u,v;q)\coloneqq\int_\Omega\nabla u\cdot\nabla v+quv\,\mathrm d\bx,\quad\ell(v)\coloneqq\int_\Omega fv\,\mathrm d\bx\quad\text{for }v\in V
		\end{align*}
		it follows that the weak formulation of \eqref{PDE1} can be expressed in the form \eqref{eq: VarState}.
		\item[(ii)] \textit{Reconstruction of the diffusion coefficient}: Let $\qa\in L^\infty(\Omega)$ be a given bound with \ReB{$\essinf_{\Omega}q_\mathsf a>0$}. We set $\Q=H^2(\Omega)\hookrightarrow C(\overline\Omega)$ and $\Qad$ is given as in \eqref{Qad}. For given $q\in\Qad$ and $f\in H$ we consider the following elliptic problem
		\begin{align}
			\label{PDE2}
			-\nabla\cdot\big(q(\bx)\nabla u(\bx)\big)=f(\bx)\text{ f.a.a. }\bx\in\Omega,\quad u(\bx)=0\text{ f.a.a. }\bx\in\Gamma.
		\end{align}
		For
		\begin{align*}
			a(u,v;q)\coloneqq\int_\Omega q\nabla u\cdot\nabla v\,\mathrm d\bx,\quad\ell(v)\coloneqq\int_\Omega fv\,\mathrm d\bx\quad\text{for }v\in V
		\end{align*}
		the weak formulation of \eqref{PDE2} is given by \eqref{eq: VarState}.
	\end{itemize}
	In both cases Assumption~\ref{A1} is satisfied. For Assumption~\ref{Ass: forward operator}, the Fr\'echet-differentiability is clear, too. Moreover, sufficient for the lack of continuity of the inverse of $\F$ is the compactness and sequential weak closedness of the operator $\F$ according to~\cite[Propositions 9.1 and 9.2]{ClasonLecture}. Since $\C$ is compact, the compactness is clear. The weak sequential closedness has been proven in \cite{KaltenbacherNeubauerScherzer+2008} under the assumption of $H^2(\Omega)$-regularity of the state $u$. In our case, this is fulfilled according to \cite[Theorem 3.2.12]{Gri11}. As $\C$ and $\Sol$ are injective, the operator~$\mathcal F=\C\circ\Sol$ is injective as well.
\end{example}

\subsection{Iteratively regularized Gau{\ss}-Newton method}
\label{sec:2.2}

To overcome the ill-posedness of $\F$ we utilize an iteratively regularized Gau{\ss}-Newton method (IRGNM) as a regularization method. We are interested in the minimization of the discrepancy
\begin{equation}
	\label{MinHatJ}
	\min\hat J(q)\coloneqq \frac{1}{2}\,{\|\F(q)-\yd\|}_\H^2\quad\text{subject to (s.t.)}\quad q\in\Qad
\end{equation}
in a stable way with respect to the noise level $\delta$, which will be achieved by adding a Tikhonov term and by early stopping in an iteratively linearized procedure (see below). In this paper, we concentrate on local convergence properties of our methods, which is often done if one is interested in the regularization; see \cite{KaltenbacherNeubauerScherzer+2008} and Remark~\ref{Rem:ConvNew1}. For this reason, we make use of the following assumption.

\begin{assumption}
	\label{AssumpMain}
	There exists a convex, closed, and bounded subset $\CQ\subset\Qad$ such that \eqref{MinHatJ} has a unique (local) minimizer $\bar q$  in the interior of $\CQ$.
\end{assumption}

\begin{remark}
	\label{Rem:Unconst}
	It follows from {\em Assumption~\ref{AssumpMain}} that the minimization problem
	\begin{equation}
		\label{Eq:IPinf}
		\min\hat J(q)\quad\text{s.t.}\quad q\in\CQ
	\end{equation}
	can be considered as a locally unconstrained problem. Thus, first-order necessary optimality conditions are $\nabla\hat J(\bar q)=0$ in $\Q$, where $\nabla\hat J$ stands for the gradient of~$\hat J$. Since we usually have $y^\mathsf e\neq\yd$, we expect also $q^\mathsf e\neq\bar q$ and therefore $\hat J(\bar q)>0$.
\end{remark}

Let $q^{k}\in\CQ$ be a given iterate which is sufficiently close to the local solution~$\bar q$. It follows from \cref{AssumpMain,Ass: forward operator} that $\F$ is continuously Fr\'echet-differentiable on $\CQ\subset\Qad$. Now, the IRGNM update scheme results from linearizing $\F$ at $q^k$ and minimizing the Tikhonov functional of the linearization
\begin{equation}
	\label{eq: IRGNMscheme_minimize}
	\tag{$\textbf{IP}^k_\alpha$}
	q(\alpha_k)=\argmin_{q\in\CQ} \frac{1}{2}\,{\|\F(q^{k})+\F'(q^{k})(q-q^{k})-\yd\|}_\H^2+ \frac{\alpha_k}{2}\,{\|q-q_\circ \|}_\Q^2,
\end{equation}
where $q_\circ\in\CQ$ is the regularization center and $\alpha_k >0$ a Tikhonov regularization parameter.

\begin{remark}
	Throughout this work, it is assumed that \eqref{eq: IRGNMscheme_minimize} possesses a unique solution $q(\alpha_k)$ which lies in the interior on $\CQ$. Therefore, \eqref{eq: IRGNMscheme_minimize} can be considered as a (locally) unconstrained problem so that we neglect the constraint $q\in\Qad$ in the following. In {\em\cite{KaltenbacherMain}}, it is shown under suitable conditions on $\F$ (see {\em Remark~\ref{rem: remark unconstrained}-(ii)}) that the iterates produced by the IRGNM stay in a ball in the interior of $\Qad$ (cf. {\em Assumption~\ref{AssumpMain}}).
\end{remark}

We accept the iterate $q^{k+1}\coloneqq q(\alpha_k)$ if $\alpha_k>0$ is chosen such that
\begin{equation}
	\label{eq: choice of alpha}
	\theta \hat J(q^k)\leq {\|\F'(q^{k})(q(\alpha_k)-q^{k})+\F(q^{k})-\yd\|}_\H^2 \leq \Theta \hat J(q^k),
\end{equation}
where $0<\theta<\Theta<1$ holds. Condition \eqref{eq: choice of alpha} can be obtained via an inexact Newton method as in \cite{KaltenbacherMain}, whereas we use a backtracking technique (see Algorithm~\ref{alg1: IRGNM}): If the first inequality in \eqref{eq: choice of alpha} is not fulfilled, then $\alpha_k$ is too small when solving \eqref{eq: IRGNMscheme_minimize}, so we increase $\alpha_k$ and solve \eqref{eq: IRGNMscheme_minimize} again. Similarly, we decrease $\alpha_k$ if the second inequality in \eqref{eq: choice of alpha} is not fulfilled.
\begin{algorithm}
	\caption{(IRGNM)}\label{alg1: IRGNM}
	\begin{algorithmic}[1]
		\Require Noise level $\delta$, discrepancy parameter $\tau>1$,  initial guess $q^0$,  initial regularization $\alpha_0>0$, regularization center $q_\circ$.
		\State Initialize $k=0$.
		\While{$\|\F(q^{k})-\yd\|_\H > \tau \delta$}
		\State Solve subproblem \eqref{eq: IRGNMscheme_minimize} for $q(\alpha_k)$.
		\While{\eqref{eq: choice of alpha} is not fulfilled}
		\State Increase (decrease) $\alpha_k$ if the first (second) part of \eqref{eq: choice of alpha} is not valid.
		\State Solve subproblem \eqref{eq: IRGNMscheme_minimize} for $q(\alpha_k)$.
		\EndWhile
		\State Set $q^{k+1}=q(\alpha_k)$, $k=k+1$.
		\EndWhile
	\end{algorithmic}
\end{algorithm}
If $\yd\notin \mbox{range}(\F)$, the iteration cannot converge, but it is possible to develop early stopping criteria to prevent noise amplification. For this we use the so-called \emph{discrepancy principle}: we stop the iteration after $k_*(\delta,\yd)$ steps provided
\begin{equation}\label{eq: discrepancy principle}
	\hat J(q^{k_*(\delta,\yd)})\leq \frac{1}{2}(\tau \delta)^2 \leq\hat J(q^k), \quad\ k=0,...,k_*(\delta,\yd).
\end{equation}
The parameter $\tau>1$ reflects the fact that we can not expect the discrepancy to be lower than the noise in the given data. Also, condition \eqref{eq: choice of alpha} can be understood as a linearized discrepancy principle. Note that the iterates depend on the noise level ($q^k=q^{k,\delta}$), but we suppress the superscript $\delta$ for readability.
\begin{remark}\label{rem: remark unconstrained}
	\begin{enumerate}
		\item [\em (i)] Let us turn back to {\em Example~\ref{ex: example_problems}-(ii)}. Since we do not want to consider $H^2$-regular coefficients in our numerical experiments, we assume that piecewise linear finite elements can be used to approximate the coefficients; see also {\em\cite{RBL,Hanke_1997,Vol02}}.
		\item [\em (ii)] A local convergence analysis of an adaptive discretization of the IRGNM is given in {\em\cite{KaltenbacherMain,diss_kirchner}}. It turns out that a tangential cone condition and a weak closedness condition on the forward operator $\F$ play an essential role in the analysis. For our particular examples these conditions were verified in {\em\cite{Hanke_1997,HaNeSc95,KaltenbacherNeubauerScherzer+2008}}. However, in this work, we do not focus on the theoretical foundations 
		in the sense of regularization theory
		but rather concentrate on the numerical realization of our reduced-order approach while assuming that our iterates converge.
	\end{enumerate}
\end{remark}
In the following, we will need the gradient of the discrepancy $\hat J$ and the optimality conditions for \eqref{eq: IRGNMscheme_minimize}, which are stated in the following remark.
\begin{remark}
	\label{Rem:FO}
	\begin{enumerate}
		\item[\em (i)] For $q\in\CQ$ the gradient $\nabla\hat J(q)\in \Q$ of $\hat J$ satisfies 
		\begin{equation}
			\label{eq: gradient}
			{\langle \nabla\hat J(q), d\rangle}_\Q = {\langle \mathcal B_u'p, d\rangle}_{\Q',\Q} = \partial_q a(u,p;d) \quad \text{for all }d\in \Q,
		\end{equation}
		where $u=\Sol(q)$ holds and $p=p(q)\in V$ is the solution of the adjoint equation
		\begin{equation}
			\label{eq: adjoint equation}
			a(v,p;q)=-\,{\langle \mathcal C^*(\mathcal Cu-\yd),v\rangle}_{V} \quad \text{for all }v\in V
		\end{equation}
		with the adjoint operator $\mathcal C^*:\H\to V$ of $\mathcal C$. For later use, we introduce the linear and continuous operator
		\begin{equation*} 
			{\langle\mathcal B_u d, v\rangle}_{V',V} = \partial_q a(u,v;d) \quad \text{for all }(d,v)\in \Q\times V.
		\end{equation*} 
        In \eqref{eq: gradient} we denote by $\mathcal B_u': V \to \Q'$ the dual operator of $\mathcal B_u$.
		\item[\em (ii)]
		Suppose that {\em Assumptions~\ref{Ass: forward operator}} and {\em\ref{AssumpMain}} hold. Let $u^k=\mathcal S(q^k)\in V$ be the state at the current iterate $q^k\in\CQ$. Due to $\alpha_k>0$ problem \eqref{eq: IRGNMscheme_minimize} is a linear-quadratic, strictly convex optimization problem with a unique (unconstrained) minimizer; see {\rm \cite[Theorem 2.14]{Troltzsch}}. The update $q^{k+1}\in\CQ$ is optimal for \eqref{eq: IRGNMscheme_minimize} if and only if there exists the linearized state $\ulin^{k+1}\in V$ and the linearized adjoint state $\plin^{k+1}\in V$ satisfying the optimality system
		\begin{subequations}
			\label{eq: optimality system subproblem}
			\begin{align}
				a(\ulin^{k+1},v;q^k)+{\langle\mathcal B_{u^k}(q^{k+1}-q^{k}), v\rangle}_{V',V}&=0 && \text{for all } v\in V,\\
				a(v,\plin^{k+1};q^k)+{\langle \C(u^k+\ulin^{k+1})-\yd, \C v \rangle}_\H&=0 && \text{for all } v\in V,\\
				{\langle\alpha_k (q^{k+1}-q_\circ) + \mathcal J_\Q^{-1}\mathcal B_{u^k}'\plin^{k+1},q\rangle}_\Q&=0&&\text{for all }q\in \Q,\label{eq: linearized gradient cond}
			\end{align} 
		\end{subequations}
		where $\mathcal J_\Q:\Q\to \Q'$ is the Riesz-isomorphism.
		\item[\em (iii)] Note that \eqref{eq: linearized gradient cond} implies the relationship
		\begin{align}
			\label{FO-q}
			q^{k+1}=q_\circ - \frac{1}{\alpha_k}\mathcal J_\Q^{-1}\mathcal B_{u^k}'{\plin^{k+1}}\quad\text{in }\Q.
		\end{align}
		In {\em Section~\ref{sec: QFOM}} we will essentially use \eqref{FO-q} to adaptively build finite-dimensional (reduced) parameter spaces for the iterates $\{q^k\}_{k\ge0}$, where we utilize $q_\circ$ as well as the ansatz spaces for the states and the dual solutions.
	\end{enumerate}
\end{remark}
%
\section{Adaptive RB schemes for infinite-dimensional parameter spaces}
\label{section: TRRB}
\setcounter{equation}{0}
Every step of the IRGNM consists of solving the linear-quadratic PDE-con\-strained optimization problem \eqref{eq: IRGNMscheme_minimize} on $\Q$ at least once. Thus, Algorithm~\ref{alg1: IRGNM} can be considered a many-query context for the parameterized equation \eqref{eq: VarState}. Hence, using the RB method to speed up the solution process seems natural.  Another reason for applying model reduction methods in the context of inverse problems are situations where the noise level is so high, such that only a vague reconstruction of the parameter is possible.

The fundamental problem in applying the RB method in our context is that the parameter space $\Q$ is infinite-dimensional and has to be discretized itself. In a discrete setting with $\Q_h =\mathrm{span}\,\{\phi_1,...,\phi_{N_\Q}\} $, one would replace the infinite-dimensional parameter with its finite element type approximation
\begin{align}
	\label{eq:qGalerkin}
	q=\sum_{i=1}^{N_\Q}\Qcoeff_i \phi_i\in \Q_h
\end{align}
with coefficients $\Qcoeff_1,\ldots,\Qcoeff_{N_\Q}\in \R^{N_\Q}$ to obtain a parameter affine decomposition of the bilinear form $a$.

For instance, if we can write the bilinear form $a$ as
\begin{align}
	\label{eq:SplitBilForm}
	a(u,v;q)=a_1(u,v)+a_2(u,v;q)\quad\text{for }u,v\in V\text{ and }q\in\Q
\end{align}
and use \eqref{eq:qGalerkin}, we derive the affine decomposition (cf. Assumption~\ref{A1}) of $a$ as
\begin{equation}
	\label{eq: affine_decomposition}
	a(u,v;q)=a_1(u,v)+\sum_{i=1}^{N_{\Q}} \Qcoeff_i \, a_2(u, v; \phi_i) \quad \text{for all }u,v \in V.
\end{equation}
Note that \eqref{eq:SplitBilForm} holds for our state equations introduced in Example~\ref{ex: example_problems}.

Importantly, the affine decomposition will consist of an arbitrarily large number $N_{\Q}\in \N$ of affine components determined by the discretization density of the space~$\Q_h$. This leads to the following known challenges in RB methods:
\begin{itemize}
	\item[(i)] A high number of affine components need to be projected onto the reduced space.
	\item[(ii)] The assembly of residual-based error estimates is infeasible because it involves the computation of Riesz representatives for each affine component per reduced basis element.
\end{itemize}
In what follows, we resolve these issues and derive an RB method for the IRGNM in two steps:
\begin{itemize}
	\item \textbf{Step~1:} We first concentrate on a problem-specific dimension reduction of the infinite-dimensional parameter space $\Q$.
	\item \textbf{Step~2:} We propose a classical RB approximation of the primal and dual state space $V$ on top of the reduced parameter space that overcomes both of the above-mentioned problems (i) and (ii).
\end{itemize}
%
\subsection{Step 1: Reduction of the parameter space}
\label{sec: QFOM}
%
Since the parameter space in our case is generally infinite-dimensional, we aim at reducing the number of affine components first and to adaptively construct a low-dimensional reduced space $\Q_r=\text{span}\,\{\upphi_1,...,\upphi_{n_\Q} \}\subset \Q$ with dimension $n_{\Q} \ll N_{\Q}$. For $q=\sum_{i=1}^{n_{\Q}} \mathsf q_i\upphi_i\in \Q_r$ and a bilinear form $a$ satisfying Assumption~\ref{A1} and \eqref{eq:SplitBilForm} we obtain a low-dimensional representation
\begin{equation}
	\label{eq: small_affine_decomposition}
	a(u,v;q) = a_1(u,v)+\sum_{i=1}^{n_{\Q}} \mathsf q_i a_{2,i}(u,v)\quad\text{for all }u,v \in V
\end{equation}
with $a_{2,i}(u,v)\coloneqq a_2(u,v;\upphi_i)$ for $i=1,...,n_\Q$. As a motivation for how to choose the snapshots for the parameter space, we recall \cref{Rem:FO} and therefore consider the first-order optimality condition of the regularized discrepancy
\begin{align*}
	\hat J_{\alpha}(q)\coloneqq\hat J(q)+\frac{\alpha}{2}\,{\|q-q_\circ\|}^2_\Q\quad\text{for }\alpha>0,\,q_\circ\in\Q\text{ and }q\in\Qad.
\end{align*}
Suppose $\bar q$ is a local unconstrained minimizer of $\hat J_\alpha$, then it holds using \eqref{eq: gradient}
\begin{align}
	\label{eq: grad}
	\nabla\hat J_{\alpha}(\bar q)=\alpha_k (\bar q-q_\circ)+ \mathcal J_\Q^{-1}\mathcal B_{\bar u}'p(\bar q) = 0;
\end{align}
compare also \eqref{FO-q}. It follows from \eqref{eq: grad} that the optimal parameter $\bar q$ is in~$\Q_r$ provided $q_\circ \in \Q_r$ and $\mathcal J_\Q^{-1}\mathcal B_{\bar u}'p(\bar q)\in\Q_r$. If one assumes a low-dimensional RB representation of the adjoint state the upper optimality condition connects the reduced basis of the parameter space with the reduced basis of the adjoint state. Therefore, we will adaptively enrich the reduced parameter space with the gradients $\nabla \hat J(q^k) =\mathcal J_\Q^{-1}\mathcal B_{u^k}'p(q^k)$ of the iterates $q^k$. 

The IRGNM with adaptive parameter space reduction (cf. Algorithm~\ref{alg: FOM IRGNM Q}) consists of outer and multiple inner iterations. In the outer iteration, we construct the reduced parameter space and in every inner iteration, we solve a subproblem with the IRGNM on the current reduced parameter space. Suppose we are given an iterate $q^k\in \CQ$ sufficiently close to the solution. The current reduced parameter space~$\Q_r^k$ contains the current iterate $q^k$, the regularization center~$q_\circ$, and the current gradient $\nabla\hat J(q^k)\in\Q$ as a basis function. Then, instead of \eqref{Eq:IPinf} we solve the low-dimensional minimization problem
\begin{equation}
	\label{eq: IRGNM parameter space subproblem}
	\min\hat J(q)\quad\text{s.t.}\quad q\in\CQ\cap\Q_r^k
\end{equation}
using the IRGNM to get the next iterate $q^{k+1}\in \Q_r^k$. Here we suppose the following hypothesis (cf. Assumption~\ref{AssumpMain} and Remark~\ref{Rem:Unconst}):

\begin{assumption}
	\label{AssumpMain2}
	For any $k\in\mathbb N$ problem \eqref{eq: IRGNM parameter space subproblem} has a unique (local) minimizer $\bar q^{k+1}$ in the interior of $\CQ\cap\Q_r^k$. Moreover, $\hat J$ is continuously Fr\'echet-differentiable in a neighborhood of $q^{k+1}$.
\end{assumption}

\begin{remark}
	\label{Rem:Unconst2}
	It follows from {\em Assumption~\ref{AssumpMain2}} that \eqref{eq: IRGNM parameter space subproblem} can be considered as a locally unconstrained minimization problem so that the constraint $q\in\CQ$ can be neglected.
\end{remark}

As a stopping criterion for \eqref{eq: IRGNM parameter space subproblem}, one can use a  discrepancy principle (compare \eqref{eq: discrepancy principle}) with a possibly modified noise level $\tilde \delta_k\geq \delta$. To update the parameter space, we compute the new gradient $\nabla\hat J(q^{k+1})$ and construct the space $\Q_r^{k+1}$ by adding $\nabla\hat J(q^{k+1})$ to $\Q^k_r$ and subsequent (Gram-Schmidt) orthonormalization. The initial regularization parameter~$\alpha_0^k$ of the inner IRGNM at outer step $k$ is chosen as follows: $\alpha_0 \coloneqq \alpha_0^0>0$ and $\alpha_0^k\coloneqq\alpha_1^{k-1}$, where $\alpha_1^{k-1}$ is the first accepted inner regularization parameter satisfying \eqref{eq: choice of alpha} of the previous outer iteration $k-1$. This procedure is repeated until the convergence criterion \eqref{eq: discrepancy principle} is fulfilled. We summarize the method in Algorithm~\ref{alg: FOM IRGNM Q}.
\begin{algorithm}
	\caption{(Parameter space reduced IRGNM: \QFOM)}\label{alg: FOM IRGNM Q}
	\begin{algorithmic}[1]
		\Require Initial guess $q^0\in\Qad$, discrepancy parameter $\tau>1$, noise level $\delta$, initial regularization parameter $\alpha_0^0$, regularization center $q_\circ\in \Q$, maximum number of inner IRGNM iterations $l^{\text{inner}}_{\text{max}}$.
		\State Set $k=0$ and initialize $\Q_r^0=\mbox{span}\{q_\circ, q^0,\nabla\hat J(q^0)\}$ by orthonormalization.
		\While{$\|\F(q^{k})-\yd\|_\H > \tau \delta$}
		\State Solve subproblem \eqref{eq: IRGNM parameter space subproblem} for $q^{k+1}$ using IRGNM with initial regulariza-
		\State tion $\alpha_0^k$, modified noise level $\tilde \delta_k$ and maximum iterations $l^{\text{inner}}_{\text{max}}$.
		\State  Update $\Q_r^k$ using $\nabla\hat J(q^{k+1})$ and extend the affine decomposition \eqref{eq: small_affine_decomposition}.
		\State Update $\alpha_0^{k+1}$, $\tilde \delta_{k+1}$ and set $k=k+1$. 
		\EndWhile
	\end{algorithmic}
\end{algorithm}

Another reason why the gradient is expected to be a good choice for the enrichment of the parameter space is that it provides information in the direction of the steepest descent of the objective~$\hat{J}$, i.e., we have $q^k-t_k\nabla\hat J(q^k)\in \Q_r^k$ for any stepsize $t_k > 0$, which implies a sufficient decrease of the objective $\hat J$. After a few iterations, the space $\Q_r^k$ contains also approximate information about the curvature in the form $\nabla\hat J(q^{k})-\nabla \hat J(q^{k-1})$ and $q^{k}-q^{k-1}$. This is used in Quasi-Newton methods ensuring superlinear local convergence. In this way, one can interpret Algorithm~\ref{alg: FOM IRGNM Q} as a subspace optimization method (cf. \cite{CSIAM-AM-2-585}; see also \cite{Wald2018} in an inverse problems context), that collects a set of search directions and computes the exact step length in each iteration. 

\begin{remark}
	\label{rem: other snapshts for Q}
	Other snapshot choices for the parameter space are possible:
	\begin{itemize}
		\item local or coarse snapshots consisting of indicator functions of some patches of the underlying mesh,
		\item curvature information by using an approximate solution of the linear-quadratic subproblem \eqref{eq: IRGNMscheme_minimize} or hessian products of $\hat J_{\alpha}$ in certain directions,
		\item gradients $\nabla\hat J(q)$ for $q\in \mathrm{span}\,\{q^{k}-q^{k-1}, \nabla\hat J(q^{k})-\nabla\hat J(q^{k-1})\}$, where the gradient serves as a non-linear mechanism to generate new information for $\Q_r^k$,
		\item subspace optimization techniques to construct $\Q_r^k$ (see {\em\cite{CSIAM-AM-2-585}}).
	\end{itemize}
\end{remark}
\noindent Our described procedure serves two major purposes:
\begin{itemize}
	\item[(i)] Reducing the number of unknowns for the IRGNM substantially such that \eqref{eq: IRGNMscheme_minimize} does not operate on $\Q$ (or $\Q_h$) but on the low dimensional space~$\Q_r$.
	\item[(ii)] Reducing the number of affine components from $N_\Q$ in \eqref{eq: affine_decomposition} to $n_\Q$ in~\eqref{eq: small_affine_decomposition}.
\end{itemize}
Point (i) alone is of interest in terms of regularization by discretization even if no additional state reduction is performed. Hence, we also investigate this method in the numerical experiments. However, since one has to solve repeatedly a FOM subproblem \eqref{eq: IRGNM parameter space subproblem}, the method will not pay off in terms of computational time unless the dimension of the parameter space dominates the cost of the FOM algorithm. If this is not the case, we suggest solving the subproblems \eqref{eq: IRGNM parameter space subproblem} inexactly by using a reduced-order approximation $\hat J_r$ of $\hat J$.
Significantly, Point (ii) is the reason why we can use an efficient RB approximation for the state space $V$. This approach is explained in the next section. 

\subsection{Step 2: Reduction of the state space}
\label{sec: QRB}

If a reduced parameter space $\Q_r$ is available, i.e., we have the low-dimensional affine representation \eqref{eq: small_affine_decomposition} of the bilinear form $a$ on $\Q_r$, we can efficiently apply the standard RB method to reduce the space $V$ for parameters in $\Q_r$. To this end, let a reduced state space $V_r$ of dimension ${n_V}\in \N$ be given. Given a parameter $q\in \Q_r\cap\Qad$, the state and parameter RB approximation of the state is given as $u_r=u_r(q)\in V_r$
\begin{equation}
	\label{eq: VarState_red}
	a(u_r,v;q) =\ell(v) \quad \text{for all }v \in V_r.
\end{equation}
We define the RB solution operator
\begin{align*}
	\Sol_r:\Q_r\cap\Qad\to V_r,\quad\Q_r\cap\Qad\ni q \mapsto \Sol_r(q)=u_r(q)\in V_r.
\end{align*}
Furthermore, the RB forward operator $\F_r$ and the RB discrepancy $\hat J_r: \Q_r\cap\Qad \to \R$ are introduced by
\begin{equation}
	\hat J_r(q) := \frac{1}{2}\,{\|\F_r(q)-\yd\|}_\H^2\quad\text{with }\F_r=\mathcal C\circ \mathcal S_r
\end{equation}
so that we solve
\begin{equation}
	\label{eq:IProm}
	\min\hat J_r(q)\quad\text{s.t.}\quad q\in\CQ\cap\Q_r
\end{equation}
instead of \eqref{eq: IRGNM parameter space subproblem}. Similar to Assumption~\ref{AssumpMain2}, we suppose that \eqref{eq:IProm} possesses a unique (local) solution $\bar q_r$ in the interior of $\CQ\cap\Q_r$. Let us mention that later both $\Q_r$ and $V_r$ are enriched and therefore depend on the iteration counter $k$.

Similar reduction schemes are used to compute the derivatives of $\F_r$ and $\hat J_r$. For instance, the adjoint $p_{r}=p_r(q)\in V_r$ is given as the solution of
\begin{equation*}
	a(v,p_{r};q)=-\,{\langle \mathcal C^*(\mathcal Cu_{r}-\yd),v\rangle}_{V_r} \quad \text{for all }v\in V_r
\end{equation*}
and the gradient of the reduced discrepancy can be expressed as
\begin{equation}
	\nabla\hat J_r(q) = \mathcal J_\Q^{-1}\mathcal B_{u_r}{p_r} \in \Q_r.
\end{equation}

To control the RB error of $\hat J_r$ (w.r.t. to the state space reduction), we need an a-posteriori error estimate.
For fixed $q\in \Q_r$, $u\in V$ we define the primal residual $r_{pr}(u;q)\in V'$ as	
\begin{equation*}
	r_{pr}(u;q)[v]\coloneqq\ell(v)-a(u,v;q)\quad \text{for all }v\in V,
\end{equation*}
whereas the adjoint residual is given for $p\in V$ as
\begin{equation*}
	r_{du}(u,p;q)[v]\coloneqq -\,{\langle\mathcal C^*(\mathcal Cu-\yd,v\rangle}_V - a(v,p;q)\quad\text{for all }v\in V.
\end{equation*}
With these definitions, we can formulate the error estimator for the discrepancy~$\hat J$.
\begin{proposition}[A-posteriori error estimate for $\hat J$] \label{prop: error estimator J}
	Let $q\in \Q_r\cap \Qad$ and $\coercivityconstant_q>0$ be the ($q$-dependent) coercivity constant for $a(\cdot,\cdot\,;q)$. Then:
		\begin{equation}\label{eq: error estimator}
			|\hat J(q)-\hat J_r(q) |\leq \Delta_{\hat J}(q),
		\end{equation}
		with
		\begin{align*}
			\Delta_{\hat J}(q)&\coloneqq \frac{{\|\mathcal C\|}_{\mathscr L(V,\H)}^2}{2}\Delta_{pr}(q)^2+ {\|r_{du}(u_{r},p_{r};q) \|}_{V'} \Delta_{pr}(q),\\
			\Delta_{pr}(q)&\coloneqq \frac{1}{\coercivityconstant_q}\, {\|r_{pr}(u_{r};q) \|}_{V'}.
		\end{align*}
		%
\end{proposition}

\begin{proof}
	The proof is analoguous to the error estimation in \cite[Proposition~3.6]{Keil2021nonco-54293}. 
\end{proof}
To evaluate the error bound $\Delta_{\hat J}$ efficiently online an offline/online decomposition of residual is necessary, which involves the computation of Riesz-representatives (see \cite{book}) of the order of $n_\Q n_V$. Furthermore, projecting the operators \eqref{eq: VarState_red} onto $V_r$ is also dependent on the number of affine coefficients. Hence, a low number of affine components $n_\Q$ is crucial.

\subsection{Parameter and state space RB trust region IRGNM}
\label{sec: TRIRGNM}

The efficient error bound on $\Q_r$ from Section~\ref{sec: QRB} enables to develop a certification of the RB model of the state equation by an error-aware trust region strategy. In this section, we combine ideas from the trust region reduced basis (TRRB) method from \cite{Keil2021nonco-54293,keil2022adaptive,QGVW17} and the strategies presented in the last two subsections to (i) cope with the infinite-dimensional parameter space and (ii) solve the ill-posed problem \eqref{eq: IP} with an adaptive IRGNM in a trust (TR) region setting. For~(i) we adaptively enrich the reduced parameter space~$\Q_r$ with the gradient of the current iterate (as described in Section~\ref{sec: QFOM}). For~(ii), we use a parameter and state space reduced inner IRGNM instead of a state space reduced Newton method as in \cite{Keil2021nonco-54293}. In particular, we only enforce an error-aware TR criterion on the state space.

Just as the IRGNM method with pure parameter space reduction, the parameter and state space reduced TR-IRGNM (\QTRRB) consists of outer and multiple inner iterations. In every outer iteration $k$, we enrich the reduced parameter space $\Q_r^k$ as well as the reduced state space $V_r^k$. The parameter space reduction $\Q_r^k$ simplifies the affine decomposition enabling the state space RB method to be efficient. Let Assumption~\ref{AssumpMain} be valid and $q^0\in \CQ\subset\Qad$ be sufficiently close to the (unconstrained) solution $\bar q$. The initial reduced state space~$V_r^0$ contains the primal and dual states $u(q^0), p(q^0)\in V$. Since the adjoint state is already computed, the gradient $\nabla\hat J(q^k)$ is cheaply available and is added to the initial reduced parameter space $\Q_r^0$. As before, we suppose $q^0, q_\circ\in \Q_r^0$. If the iterate $q^k$ does not satisfy the overall discrepancy principle~\eqref{eq: discrepancy principle}, we perform the following steps to update the iterate.

\subsubsection*{Computation of the AGC point.}

Throughout we assume that the iterates $q^k$ belong to $\CQ\subset\Qad$. First, we compute the approximated generalized Cauchy (AGC) point $\qagc^k\in \Q_r^k$, which can be identified in the inverse problem setting as a Landweber-type update and is defined as follows
\begin{equation}
	\label{eq: agc point}
	\qagc^k = q^k-t_k\nabla\hat J_r(q^k)\in \Q_r^k.
\end{equation} 
In \eqref{eq: agc point} the scalar $t_k>0$ is a stepsize chosen to satisfy $\qagc^k\in\Qad$, a sufficient decrease condition in the reduced objective 
\begin{equation}
	\label{eq: armijo 1}
	\hat J_r(\qagc^k)-\hat J_r(q^k)\leq - \frac{\kappa_{\mbox{\tiny arm}}}{t_k}\,{\|q^k - \qagc^k \|}_\Q^2,
\end{equation}
for $\kappa_{\mbox{\tiny arm}}>0$ and the TR condition 
\begin{equation}
	\label{eq: armijo 2}
	\mathcal{R}^k_{\hat J}(\qagc^k)\leq \TRradius^{(k)}.
\end{equation}
Here $\TRradius^{(k)}>0$ denotes the current TR radius and the relative error estimator is defined as
\begin{equation*}
	\mathcal{R}^k_{\hat J}(q)\coloneqq \frac{\Delta_{\hat J}(q)}{\hat J_r(q)} \quad\text{for }q\in \Q_r^k\cap\Qad
\end{equation*} 
with $\Delta_{\hat J}(q)$ from Proposition~\ref{prop: error estimator J}.
The Landweber step/AGC point ensures a sufficient decrease in every iteration and serves as an initial guess for the subproblem solver. Conditions \eqref{eq: armijo 1} and \eqref{eq: armijo 2} are found via a backtracking strategy, which is initialized for the computation of the AGC point as $\ReA{t} = 0.5\| \nabla\hat J(q^0)\|_\Q^{-1}$, which is the step length that was used in \cite{RBL,KaltenbacherNeubauerScherzer+2008} for the Landweber algorithm. \ReB{Note that it is always possible to find a step size $t_k>0$ with \eqref{eq: armijo 2} by the exactness of the ROM at the current iterate $q^k$ and the continuity of the error estimator $\Delta_{\hat{J}}.$}
%
\subsubsection*{The TR subproblem.}

We compute the trial step $\qtrial \in \CQ\cap\Q_r^k$ by solving the parameter and state space-reduced problem
\begin{equation}
	\label{eq: TR subproblem}
	\tag{$\textbf{IP}^k_r$}
	\min_{q\in\CQ\cap\Q_r^k}\hat J_r(q) 
	\quad\text{s.t.}\quad\mathcal{R}^k_{\hat J}(q)\leq \TRradius^{(k)}.
\end{equation}
Analogously to Remark~\ref{Rem:Unconst} we assume that \eqref{eq: TR subproblem} admits a locally unique (unconstrained) solution in $\Q_r^k$ so that we can neglect the convexity constraint $q\in\CQ$ in the numerical solution method. Thus, \eqref{eq: TR subproblem} can be solved using the IRGNM with Tikhonov regularization (see Algorithm~\ref{alg1: IRGNM}). Again, the additional TR constraint is treated using an Armijo-type backtracking technique (see \cite{keil2022adaptive, QGVW17, Banholzer2020AnAP}), which means the error estimator $\Delta_{\hat J}$ has to be cheaply available. For the initialization of the Armijo backtracking we use the initial step length $\ReA{t} =1$ in this case. As an initial guess, we use the AGC point $\qagc^k$, which is defined in \eqref{eq: agc point}. The initial regularization parameter $\alpha_0^k$ is chosen as in Section~\ref{sec: QFOM}.
Note that we only control the approximation quality of the state-space RB with $\Delta_{\hat J}$ which also contains approximation information about the primal and the dual state. However, apart from the dual model, we do not control derivative information which has shown advantageous in terms of computational efficiency in former works (see e.g. \cite{Banholzer2020AnAP}). \ReA{Let $q^{k,l}\in \Q_r^k$ be the $l$-th iterate of the inner IRGNM at outer iteration $k$}. As stopping criteria for the subproblem solver, we use on the one hand a reduced discrepancy principle 
\begin{equation}
	\label{eq: TR stopping 1}
	{\|\F_r(\ReA{q^{k,l}})-\yd\|}_\H \leq \tilde{\tau} \tilde{\delta}_k
\end{equation}
with a (possibly) modified noise level $\tilde{\delta}_k\geq \delta$, e.g. $\tilde{\delta}_k=\hat J(q^k)$ and \ReA{$\tilde{\tau}\geq \tau$ with $\tau$ from \eqref{eq: discrepancy principle}}. On the other hand, we use the usual TR stopping condition (see \cite{Keil2021nonco-54293,keil2022adaptive,QGVW17})
\begin{equation}
	\label{eq: TR stopping 2}
	\beta_1 \TRradius^{(k)}\leq\mathcal{R}^k_{\hat J}(\ReA{q^{k,l}})\leq \TRradius^{(k)},
\end{equation}
where $\beta_1\in (0,1)$ is close to one to \ReB{prevent the inner solver iterating  near the TR boundary, where the ROM becomes inaccurate. We denote by $\qtrial$ the result of the inner IRGNM satisfying either \eqref{eq: TR stopping 1} or \eqref{eq: TR stopping 2}.}

\subsubsection*{Acceptance of the trial step and modification of the TR radius.}

We consider cheaply available sufficient and necessary optimality conditions for the acceptance of the trial step $\qtrial$ (see \cite{QGVW17, Banholzer2020AnAP}). A sufficient condition for acceptance is the following
\begin{equation}
	\label{eq: TR sufficient crit}
	\hat J_r(\qtrial)+\Delta_{\hat J}(\qtrial)< \hat J_r(\qagc^k)
\end{equation}
and a necessary condition is
\begin{equation}
	\label{eq: TR necessary crit}
	\hat J_r(\qtrial)-\Delta_{\hat J}(\qtrial)\ReB{\leq}\hat J_r(\qagc^k).
\end{equation}
\ReB{Note that the negated version of statement \eqref{eq: TR necessary crit} is sufficient for rejection of $\qtrial$ (cf. line $13$ in Algorithm \ref{alg: TRRB IRGNM Q})}. If these cheap conditions do not give information about acceptance or rejection, we use the corresponding FOM condition:
\begin{equation}
	\label{eq: TR fom crit}
	\hat J(\qtrial)\leq\hat J_r(\qagc^k).
\end{equation}
In the case of acceptance, we set  $q^{k+1}= \qtrial$ and enlarge the TR radius, if the RB objective predicts the actual reduction, i.e., we check if
\begin{equation}
	\label{eq: TR reduction coeff}
	\varrho^{(k)}\coloneqq \frac{\hat J(q^{k})-\hat J(q^{k+1})}{\hat J_r(q^{k})-\hat J_r(q^{k+1})}\geq \beta_2 \in \big[\nicefrac{3}{4},1\big).
\end{equation}
Note that in case of acceptance, we enrich the spaces $\Q_r^k$ and $V_r^k$, such that the FOM quantity $\hat{J}(q^{k+1})$ is available anyway. In the case of rejection, we set $q^{k+1}= q^k$ and diminish the TR radius.

\subsubsection*{Extending the reduced spaces $\bm{\Q_r}$ and $\bm{V_r}$.}
In case of acceptance, we compute the state $u(q^{k+1})$ and check the FOM discrepancy principle. If this is not fulfilled, we compute the adjoint state $p(q^{k+1})$ and the FOM gradient $\nabla\hat J(q^{k+1})$. 
Then, we enrich the space~$\Q_r^{k}$ by orthonormalization and enlarge the affine decomposition. Afterward, we update the RB space $V_r^{k}$ using $u(q^{k+1})$ and $p(q^{k+1})$ by orthonormalization, and the error estimator is extended by the Riesz representatives of the new affine component and the new state space RB elements. For another combined parameter and state space enrichment strategy we refer to Remark \ref{rem: other enrichment strategies}. The method can be furthermore combined with the skip-basis-enrichment strategy from \cite{Banholzer2020AnAP}.

\subsubsection*{Assembly of the error estimator.}

The assembly of the error estimator is the most expensive part in terms of computation time and FOM solves (see Section \ref{section: numerical results}) for the computation of the Riesz representatives (cf. \cite{book}), especially if~$n_\Q$ and~$n_V$ are large. In this case, it may be more expensive to online/offline decompose the residual norms in Proposition \ref{prop: error estimator J} than to evaluate them online in the Armijo-type line search. If this is the case in iteration $k>2$, we do not assemble the error estimator anymore but compute it online in the next iteration. Formally, we stop the assembly of the error estimator in iteration $k+1$ if it holds
\begin{equation}\label{eq: error est assembly condition}
	K_{\mbox{\tiny ass}}^{k}> K_{\mbox{\tiny online}}^{k},
\end{equation}
where 
\begin{align*}
	K_{\mbox{\tiny ass}}^k\coloneqq\mbox{dim}(V_r^k)\big(\mbox{dim}(\Q_r^k)-\mbox{dim}(\Q_r^{k-1})\big)+\mbox{dim}(\Q_r^{k-1})\big(\mbox{dim}(V_r^k)-\mbox{dim}(V_r^{k-1} \big)
\end{align*}
is the number of Riesz representatives that were needed to update the error estimator from iteration $k-1$ to $k$. On the other hand, $K_{\mbox{\tiny online}}^{k}$ is the number of FOM solves, which would be needed to compute the error estimator online in iteration $k$. In detail, this is twice the number of total Armijo iterations for the computation of the AGC point and the steplength in every iteration of the inner IRGNM at outer iteration $k$, since in every Armijo iteration we have to evaluate the error estimator $\Delta_{\hat{J}}$ online, which involves the computation of the primal and dual residual norm (cf. Proposition \ref{prop: error estimator J}). \ReA{Note that verifying condition \eqref{eq: error est assembly condition} is numerically cheap because it involves only counting iterations of the backtracking procedure and checking the dimension of the reduced spaces. For a detailed study of this condition, we refer to Section \ref{subsec: error est case study}.} The whole procedure of the parameter and state space reduced method is depicted in Algorithm \ref{alg: TRRB IRGNM Q}.

\begin{algorithm}
	\caption{(Parameter \& state space reduced TR-IRGNM: \QTRRB)}\label{alg: TRRB IRGNM Q}
	\begin{algorithmic}[1]
		\Require Initial guess $q^0$, discrepancy parameters \ReA{$\tilde{\tau}\geq \tau >1$}, noise level $\delta$, TR radius $\eta^{{(0)}}$, boundary parameter $\beta_1\in (0,1)$, tolerance for enlargement of the radius $\beta_2 \in [3/4,1)$, shrinking factor $\beta_3\in (0,1)$, initial regularization parameter $\alpha_0^0$, Armijo parameter $\kappa_{\mbox{\tiny arm}}>0$, regularization center $q_\circ$.
		\State Set $k=0$ and initialize the RB model at $q^0$: create $V_r^0=\mbox{span}\{u(q^0),p(q^0)\}$, $\Q_r^0=\mbox{span}\{q_\circ, q^0,\nabla\hat J(q^0)\}$ by orthonormalization.
		\While{$\|\F(q^{k})-\yd\|_\H > \tau \delta$}
		\State Compute AGC point $\qagc^k$ according to \eqref{eq: agc point}.
		\State Solve \eqref{eq: TR subproblem} 
		using IRGNM with stopping criteria \eqref{eq: TR stopping 1}, \eqref{eq: TR stopping 2} for $\qtrial$ 
		\State with initial regularization $\alpha_0^k$.
		\If{\eqref{eq: TR sufficient crit}} 
		\State Accept trial step $q^{k+1}=\qtrial$, update $V_r^k$ at $u(q^{k+1}),\  p(q^{k+1})$ and $\Q_r^k$ 
		\State at $\nabla\hat  J(q^{k+1})$.
		\State Compute $\varrho^{(k)}$ from \eqref{eq: TR reduction coeff}.
		\If{$\varrho^{(k)}>\beta_2$}
		\State Enlarge radius $\TRradius^{(k+1)}=\beta_3^{-1} \TRradius^{(k+1)}$.
		\EndIf
		\ElsIf{not \eqref{eq: TR necessary crit} }
		\State  Reject trial step, set $q^{k+1}=q^k$ and shrink radius $\TRradius^{(k+1)}=\beta_3 \TRradius^{(k+1)}$.
		\Else 
		\State Compute the FOM discrepancy $\hat J(\qtrial)$.
		\If{\eqref{eq: TR fom crit}} 
		\State Accept trial step $q^{k+1}=\qtrial$, update $V_r^k$ at $u(q^{k+1}), p(q^{k+1})$ and 
		\State $\Q_r^k$  at $\nabla\hat  J(q^{k+1})$.
		\If{$\varrho^{(k)}>\beta_2$}
		\State Enlarge radius $\TRradius^{(k+1)}=\beta_3^{-1} \TRradius^{(k+1)}$.
		\EndIf
		\Else 
		\State Reject trial step, set $q^{k+1}=q^k$, shrink radius $\TRradius^{(k+1)}=\beta_3 \TRradius^{(k+1)}$.
		\EndIf
		\EndIf
		\State In case of acceptance update $\tilde{\delta}_{k+1}$, $\alpha_0^{k+1}$ and set $k=k+1$.
		\EndWhile
	\end{algorithmic}
\end{algorithm}

\begin{remark}[Other enrichment strategies]\label{rem: other enrichment strategies}
	One can also enrich the state RB space with the sensitivities \ReA{$\ulin$} and \ReA{$\plin$} at $q^k$ in a direction $d$ (see \eqref{eq: optimality system subproblem}). Then curvature information in the form $\nabla^2 \tilde J_{\alpha}(q^k) d$ is available for an enrichment of~$\Q_r^k$, where~$\tilde J_{\alpha}$ is the linearized discrepancy cost from \eqref{eq: IRGNMscheme_minimize}. We tested this approach numerically for~$d = \nabla \hat J(q^k)$ but did not achieve a more efficient RB method for our examples.
\end{remark}

\subsubsection*{\ReA{Convergence}}\label{Rem:ConvNew1}
 \ReA{We distinguish between convergence w.r.t. the iteration counter $k$ and convergence w.r.t. the noise level $\delta$ tending to zero (the regularization). Global convergence w.r.t. to the iteration counter $k$ can be ensured for our reduced methods. For the \QFOM$\ $, every step is at least as good as a gradient descent step with exact step length (if the maximum inner iterations are chosen sufficiently large).
 On the other hand, the steps in the \QTRRB$\ $are at least as good as a gradient descent step with Armijo step length for the \QTRRB, since we are exact at the current iterate and the current gradient due to $q^k, \nabla J(q^k)\in \Q^k_r$ and $u^k,p^k\in V^k_r$. However, one cannot expect to obtain efficient methods in this case, since if the initial guess is distant from the solution, the reduced methods might require numerous iterations. Without a strategy to coarsen the basis during the iterations, the computational effort may grow enormously. However, conditions on how to achieve online coarsening are beyond the scope of this paper. Global convergence w.r.t. the iteration counter for an error-aware TRRB method for PDE-constrained parameter optimization using only state space reduction was proven in~{\rm\cite{Banholzer2020AnAP}}. For the IRGNM, a local convergence and regularization result for adaptive finite element approximations was presented in \cite{KaltenbacherMain}. Further, the authors in {\rm\cite{Wang_2005}} proved local convergence and a regularization property for ill-posed problems for a standard metric TR algorithm under similar conditions as the ones mentioned in  Remark~\ref{rem: remark unconstrained}. 
 Thus, we expect a local convergence result in the sense of regularization theory also for our method. However, this aspect is not addressed in the present paper and is subject to future work.
Note that in our case the trust region constraint is not necessarily present to ensure global convergence w.r.t. to the iteration counter, but to serve as a regularizing constraint (see \cite{Wang_2005}) and to control the ROM error.}
%
\section{Numerical experiments}
\label{section: numerical results}
\setcounter{equation}{0}

In this section, we compare the numerical performance of the proposed algorithms applied on the two scenarios from \cref{ex: example_problems}. For the sake of brevity, in this section, we call them as follows: 
\begin{enumerate}
	\item \FOM: the standard FOM IRGNM (cf. Algorithm \ref{alg1: IRGNM}).
	\item \QFOM: the parameter space reduced IRGNM (cf. Algorithm \ref{alg: FOM IRGNM Q}).
	\item \QTRRB: the combined parameter and state space reduced TR-IRGNM (cf. Algorithm \ref{alg: TRRB IRGNM Q}).
\end{enumerate}

\begin{remark}
	Let us recall that {\QFOM} is not considered as a solution method on its own, but rather as a necessary step to obtain a parameter space reduction that enables us to build an efficient \QTRRB.
\end{remark}

For the simulations we used a \texttt{Python} implementation, using pyMOR (\cite{doi:10.1137/15M1026614}) for the built-in full-order discretization and the model reduction part. 
The source code to reproduce the results is available at \cite{source_code}.
Our numerical findings were obtained on a MacBook Pro 2020 with a 2,3 GHz Quad-Core i7 and 16 GB RAM.

\subsection{Computational details}

For the following model problems, we choose the computational domain $\Omega=(0,1)^2\subset\mathbb R^2$. The FOM discretization of the spaces $\Q$, $\H$, $H$, and $V$ is realized by the space $Q_h$ spanned by piecewise bilinear finite element (FE) basis functions on quadrilateral cells with $N_{Q}= 90,601$ dofs. The noisy data $\yd$ is generated as follows: First, the exact parameter $q^\mathsf e$ is interpolated on a very fine mesh with $361,201$ dofs, then the state equation for the exact parameter is solved to obtain the exact state $u^\mathsf e$, which is interpolated on the grid used for the computations. Thereafter, we add uniformly distributed noise $\xi\in\H\setminus\{0\}$ with noise level $\delta>0$ such that
\begin{equation*}
	\yd=\mathcal Cu^\mathsf e+ \delta\,\frac{\xi}{\|\xi\|_{\H}}.
\end{equation*}
%
The noise level was set to $\delta=10^{-5}$ and the parameters for the inner and outer IRGNM are chosen as
\begin{align*}
	\theta = 0.4,\quad\Theta=0.9,\quad\tau=3.5,\quad\tilde\tau=1,\quad\tilde\delta_k= \delta. 
\end{align*}
The TR parameters are chosen as:
\begin{align*}
	\beta_1=0.95,\quad\beta_2=0.75,\quad\beta_3=0.5,\quad\eta^{(0)}=0.1,\quad \kappa_{\mbox{\tiny arm}}=10^{-12}.    
\end{align*}
In all numerical experiments, the inequality constraints in $\Qad$ are not taken explicitly into account and the optimization problems are considered as unconstrained optimization problems in $\CQ$ (cf. Assumption \ref{AssumpMain}), which was also done, e.g., in \cite{KaltenbacherMain,IK90,Vol02}. Moreover, we use a discretize-before-optimize (DBO) approach. In particular, we solve the linear-quadratic subproblems \eqref{eq: IRGNMscheme_minimize} using the conjugate gradient (CG) algorithm. For the \QFOM$\ $algo\-rithm we solve the subproblems inexactly by limiting the maximum number of inner IRGNM iterations to \ReA{$l^{\text{inner}}_{\text{max}}=2$} (cf.~Algorithm \ref{alg: FOM IRGNM Q}) since this is enough to obtain a new snapshot for the reduced parameter space, while it avoids unnecessary FOM CG iterations.
In the following numerical experiments, we consider the two scenarios from \cref{ex: example_problems} for a right-hand side $f\equiv 1$, where we reconstruct a reaction/diffusion coefficient starting from a background $q^0=q_\circ \equiv 3\in\Qad$. The lower bound in the definition of $\Qad$ in~$\eqref{Qad}$ is given as $q_\mathsf a\equiv 0.001$. In all experiments, the observation operator is the canonical embedding and we have $\|\C\|_{\mathcal L(V,\H)} = 1$ for the error estimator~$\Delta_{\hat J}$ in Proposition~\ref{prop: error estimator J}.

We compare the three proposed algorithms in terms of FOM PDE solves to measure the impact of the state space reduction and the number of FOM $B_u$ (and $B_u'$)  applications, which are needed to compute the gradient in \eqref{eq: gradient} and in each conjugate gradient iteration (see the optimality system \eqref{eq: optimality system subproblem}) to measure the impact of the parameter space reduction. Moreover, we compare the algorithms in terms of total computational time and the quality of the reconstruction.
\begin{remark}[The operator $\mathcal B_u$]
	Let $u\in V$ be fixed. We discuss the application of the operators $\mathcal B_u:\Q \to V'$ and $\mathcal B_u': V\to \Q'$ for the two cases from Example~{\rm\ref{ex: example_problems}}.
	\begin{itemize}
		\item[\em (i)] For the reaction problem we have for $q\in \Q$ and $v\in V$
		\begin{align*}
			{\langle \mathcal B_u q, v\rangle}_{V',V} = \int_{\Omega} q u v \ d\bx
		\end{align*}
		and for $p\in V$ and $v\in \Q$
		\begin{align*}
			{\langle \mathcal B_u' p, v\rangle}_{\Q',\Q} = \int_{\Omega} v u p \ d\bx.
		\end{align*}
		Since we discretize the spaces $\Q$ and $V$ using the same FE space $Q_h$, we obtain a discrete symmetric operator $B_u\in \R^{N_Q\times N_Q}$, which is the weighted mass matrix with
		\begin{align*}
			(B_u)_{i,j}=\int_{\Omega} u v_i v_j \ d\bx \quad \mbox{for } i,j=1...,N_Q.
		\end{align*}
		\item[\em (ii)] For the diffusion problem we have for $q\in \Q$ and $v\in V$
		\begin{align*}
			{\langle\mathcal B_u q,v\rangle}_{V',V} = \int_{\Omega} q \nabla u\cdot \nabla v \ d\bx.
		\end{align*}
		and for $p\in V$ and $v\in \Q$
		\begin{align*}
			{\langle \mathcal B_u' p, v\rangle}_{\Q',\Q} = \int_{\Omega} v \nabla u\cdot\nabla p \ d\bx.
		\end{align*}
		is no more symmetric and can be discretized by a matrix $B_u\in \R^{N_Q\times N_Q}$ such that
		\begin{align*}
			(B_u)_{i,j}=\int_{\Omega} v_j \nabla u \cdot \nabla v_i \ d\bx \quad \mbox{for } i,j=1...,N_Q
		\end{align*}
		and the discrete actions of $\mathcal B_u$ and $\mathcal B_u'$ can be realized on the FE level by multiplication of $B_u$ and its transpose $B_u^T$ respectively.
	\end{itemize}
\end{remark}

\subsection{Run~1: Reconstruction of the reaction coefficient -- Test~1}
\label{subsec. reconstr. reaction coefficient}

We consider the situation in \cref{ex: example_problems} (i) of reconstructing the reaction coefficient $q\in\Q=L^2(\Omega)$; cf. also \cite{diss_kirchner}. Thus, we study the parametrized problem
\begin{equation}\label{eq: reaction strong}
	\begin{aligned}
		-\Delta u(\bx)+&q(\bx) u(\bx)=f(\bx)&&\text{for all }\bx\in\Omega,\\
		u(\bx)&=0&&\text{for all }\bx\in\partial\Omega.
	\end{aligned}
\end{equation}
We choose the $H^1_0(\Omega)$-norm on $V$. The coercivity constant of the corresponding bilinear form for a parameter $q$ is explicitly given as $\coercivityconstant_q=1$ and the initial Tikhonov regularization for all methods was chosen as $\alpha_0 = 1$. Moreover, we choose the same exact parameter $q^\mathsf e$ (see \cref{fig: kirchner reaction pictures}) as in \cite{diss_kirchner}, but shifted by the background $q_\circ$, i.e.
\begin{equation*}
	q^\mathsf e= q_\circ + q_1^\mathsf e + q_2^\mathsf e,
\end{equation*}
\begin{figure}[ht!]
	\centering
    \includegraphics[scale=1]{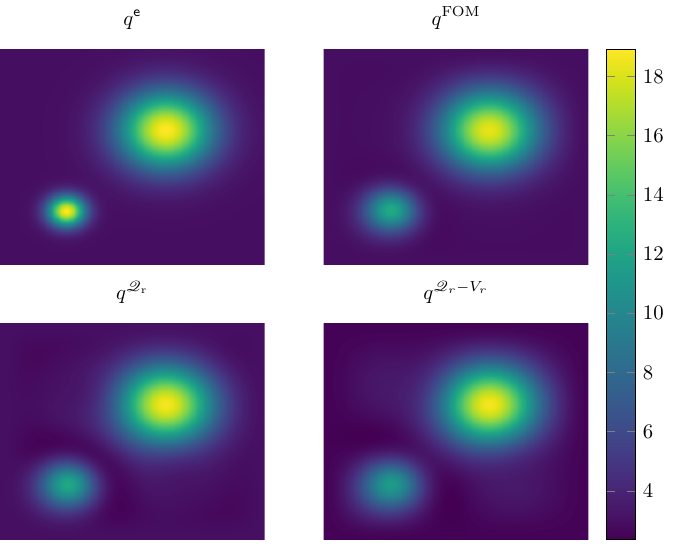}
	\caption[a]{Run~1: The exact parameter $q^\mathsf e$ and its
		three reconstructions $q^\mathrm{FOM}$, $q^\mathrm{\Q_r}$ and $q^{\Q_r\text{-}V_r}$} 
	\label{fig: kirchner reaction pictures}
\end{figure}
where the two Gaussian distributions $q_1^\mathsf e$, $q_1^\mathsf e$ are given for $x\in \Omega$ as
\begin{align*}
	q_1^\mathsf e (\bx) &=\frac{1}{2\pi\sigma^2}\exp\Bigg(  -\frac{1}{2}\Bigg( \bigg(\frac{2x_1-0.5}{0.1}\bigg)^2+ \bigg(\frac{2x_2-0.5}{0.1}\bigg)^2\Bigg)\Bigg),\\
	q_2^\mathsf e (\bx)& =\frac{1}{2\pi\sigma^2}\exp\Bigg(  -\frac{1}{2}\Bigg( \bigg(\frac{0.8x_1-0.5}{0.1}\bigg)^2+ \bigg(\frac{0.8x_2-0.5}{0.1}\bigg)^2\Bigg)\Bigg).
\end{align*}
A comparison of the algorithms in terms of computation time, FOM PDE solves and FOM $\mathcal B_u/\mathcal B_u'$ applications, reduced basis sizes, and iterations needed to reach the discrepancy threshold $\tau\delta$ is given in \cref{fig: kirchner reaction table}.
\begin{table}[ht] 
	\centering 
	\begin{tabular}{lcccccc}\toprule
		Algorithm & time [s]  & FOM solves & FOM $\mathcal B_u/\mathcal B_u'$ &$n_\Q$ & $n_V$ &  o. iter \\ \midrule
		\FOM  & 51  & \phantom{1}888 & 842 & -- &  -- & 22\\
		\QFOM  &  67 & 1131 & \phantom{1}13 &    14 & --&  12\\ 
		\QTRRB &  14 & 24(+124) & \phantom{11}8 & \phantom{1}9 & 16 &  \phantom{1}7\\ \bottomrule
	\end{tabular}
	\caption{Run~1: Comparison of the performance of all algorithms. For the \QTRRB$\ $algorithm, the number of FOM solves needed for the error estimator is in brackets.}
	\label{fig: kirchner reaction table}
\end{table}
We observe the impact of the reduction of the parameter space from the FOM $\mathcal B_u/\mathcal B_u'$ applications in the \FOM$\ $algorithm compared to the methods using the parameter space reduction, which reduced the number of FOM $\mathcal B_u/\mathcal B_u'$ applications by a factor of $100$. Additionally, the impact of the state space reduction is well visible, since the number of FOM PDE solves is significantly reduced by the \QTRRB$\ $algorithm compared to the \FOM$\ $algorithm. Also, from \cref{fig: kirchner reaction table} we deduce that the verification of the reduced order model of the error estimator is the most expensive part of the \QTRRB$\ $algorithm since 124 of the total 148 FOM solves are needed to certify the reduced order model in this case. The \QFOM$\ $algorithm is slower than the FOM method, while the \QTRRB$\ $algorithm gives a speed-up of about $3.5$. This is also depicted in the left plot in \cref{fig: kirchner reaction plots}, where the discrepancy is plotted against the CPU time.
We conclude that the parameter-reduced algorithms can reconstruct the parameter up to the discrepancy threshold by only a small reduced parameter basis of size $n_\Q=9$ ($n_\Q=14$). The gradient norm at the reconstructed parameters is of the order of $10^{-9}$. The reconstructed parameters are depicted in \cref{fig: kirchner reaction pictures}.
\begin{figure}[ht!]
	\centering
    \includegraphics[scale=1]{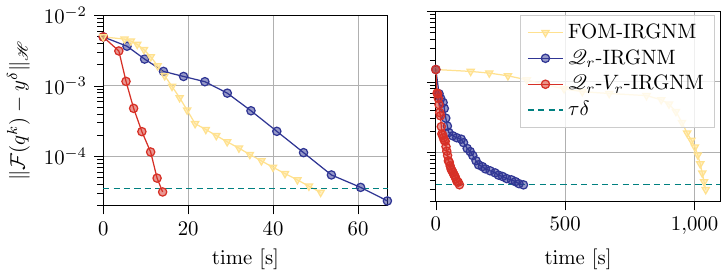}
	\caption[a]{The residuals $\|\F(q^k)-\yd\|_\H$ per computation time needed for iteration $k$ for Run~1 (left plot) and for Run~2 (right plot)} 
	\label{fig: kirchner reaction plots}
\end{figure}
All methods capture the most important characteristics of $ q^\mathsf e$, while the reconstruction of the FOM has the best quality. The relative errors are very small and given as
\begin{align*}
	\frac{\|q^\mathrm{FOM}-q^{\Q_r}\|_\Q}{\|q^\mathrm{FOM}\|_\Q}\approx 2\,\% \quad\text{and}\quad \frac{\|q^\mathrm{FOM}-q^{\Q_r\text{-}V_r}\|_\Q}{\|q^\mathrm{FOM}\|_\Q}\approx 6\,\%.
\end{align*}
\ReA{Further, we illustrate the pointwise relative errors of the reduced methods w.r.t. to the FOM in \cref{fig: error reaction 1 problem pictures}. While the approximation in the interior of the domain is lower than $10\%$, one also observes that the reconstruction of the \QTRRB$\ $at the boundary of the domain is worse than the one for the \QFOM$\ $. It is important to highlight that the parameter-reduced methods seek parameters within the
reduced parameter space that optimally fit with the data $\yd$ concerning the discrepancy principle. Nevertheless, the discrepancy principle exclusively guarantees the approximation quality only within the image of the forward operator $\F$, where the boundary information is not taken into account due to the homogeneous Dirichlet conditions. In that manner the boundary may be not well-approximated, if the reduced parameter space is of low dimension, which is the case here (cf. \cref{fig: kirchner reaction table}).}
\begin{figure}[ht]
	\hspace{0.8cm}\includegraphics[scale=1]{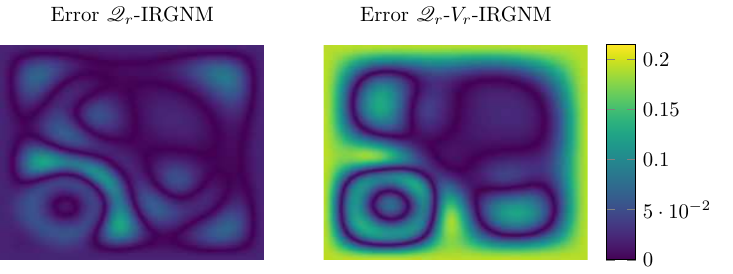}
    \caption[a]{Run~1: pointwise relative errors w.r.t. to the FOM reconstruction}
	\label{fig: error reaction 1 problem pictures}
\end{figure}
%
\subsection{Run~2: Reconstruction of the diffusion coefficient -- Test~1}\label{subsec: reconstr. diffusion coeff}
%
We consider the situation in Example~\ref{ex: example_problems} (ii), that is we consider the parametrized problem
\begin{equation}
	\label{eq: diffusion strong 2}
	\begin{aligned}
		-\nabla\cdot\big(q(\bx) \nabla u(\bx)\big) &=f(\bx)&&\text{for all }\bx\in\Omega,\\
		u(\bx)&=0&&\text{for all }\bx\in\Omega.
	\end{aligned}
\end{equation}
We choose the $H^1_0(\Omega)$-norm on $V$, which results in a coercivity constant of $\underline\alpha_q=\essinf_\Omega q$. We use the same model problem as in \cite{RBL}, where the exact parameter is given as $q^\mathsf e=q_0+\ccont\chi_{\Omega_1}-2\chi_{\Omega_2}$ (see \cref{fig: P1H1 Landweber problem pictures}) with 
\begin{align*}
	\Omega_1&=[5/30,9/30]\times [3/30,27/30] \nonumber \\
	& \quad\cup \big([9/30, 27/30]\times \big([3/30,7/30]\cup [23/30,27/30]\big) \big),\\
	\Omega_2&= \big\{x\in \Omega\,|\, \|x-(18/30, 15/30)^T \|< 4/30 \big\},
\end{align*}
and contrast parameter $\ccont=2$. 
\begin{figure}[ht]
	\centering
	\includegraphics[scale=1]{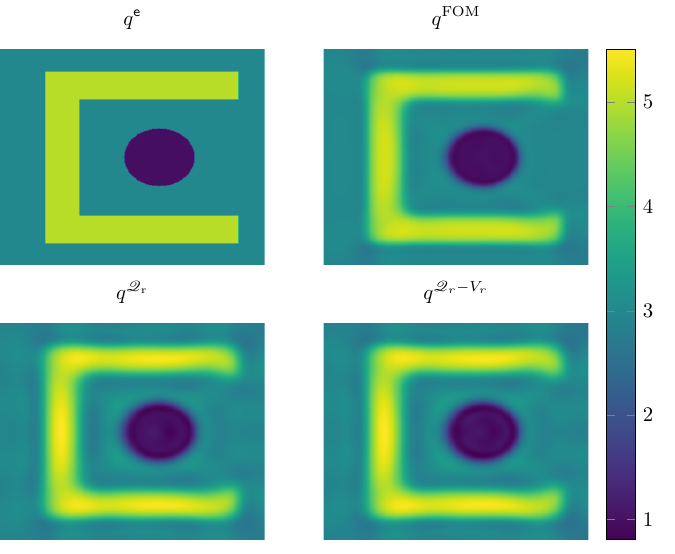}
    \caption[a]{Run~2: the exact parameter $q^\mathsf e$ and its three reconstructions $q^\mathrm{FOM}$, $q^\mathrm{\Q_r}$ and $q^{\Q_r\text{-}V_r}$}
	\label{fig: P1H1 Landweber problem pictures}
\end{figure}
In our numerical experiments, we observe that regularizing with the $L^2(\Omega)$-norm resulted in small artifacts of high amplitude on some mesh cells for the \FOM$\ $algorithm. This relates also to the fact that from a theoretical point of view, the canonical Hilbert space for $\Q$ is $H^2(\Omega)$ (see \cref{ex: example_problems} (ii)), and the canonical Banach space is $W^{1,4}(\Omega)$ (see \cite{Kaltenbacher_2009}). Since we neither want to use FE subspaces of $H^2(\Omega)$ nor work in a Banach space setting, we simply choose $\Q=H^1(\Omega)$ equipped with its canonical norm (cf. \cref{rem: remark unconstrained}). The initial Tikhonov regularization for all methods was chosen as $\alpha_0 = 10^{-3}$. The enrichment of the parameter space is done as follows. The discrete gradient of the regularized cost function $\hat J_{\alpha}$ for $\alpha >0$ is given as
\begin{equation}\label{eq:H1 gradcont discrete}
	\nabla\hat J(q) = \alpha S(q-q_\circ) + B_u^Tp\in \R^{N_\Q},
\end{equation}
where $S\in \R^{N_\Q\times N_\Q}$ is the discretized Neumann operator and $B_u^T\in \R^{N_\Q\times N_\Q}$ the discretization of the operator $\mathcal B_u'$. To improve the quality of the reconstruction for the parameter-reduced methods, we perform a smoothing of the gradient by choosing the snapshot to enrich the reduced parameter space as
\begin{equation}\label{eq: snapshot choice diffusion}
	q = S^{-1} B_u^Tp,
\end{equation}
which corresponds to the optimize-then-discretize gradient $J_\Q^{-1}\mathcal B_u'p$ from \eqref{eq: gradient} and the optimality condition \eqref{eq: optimality system subproblem}. Note that also \eqref{eq: snapshot choice diffusion} comes from setting \eqref{eq:H1 gradcont discrete} equal zero, solving for $q$ and neglecting the scaling of $\alpha$ and the term $q_\circ$, which is already in the basis.\\
A comparison of the algorithms is given in \cref{fig: P1H1 Landweber problem table} and in the right plot in \cref{fig: kirchner reaction plots}. The forward FOM solves are reduced by a factor of $4$ and $60$ for the \QFOM$\ $and the \QTRRB, respectively. Correspondingly, in our numerical experiments we observe a speedup in computation time of about 3 and 10 for the \QFOM$\ $ and the \QTRRB, respectively. 
\begin{table}[ht] 
	\centering 
	\begin{tabular}{lcccccc}\toprule
		Algorithm & time [s]  & FOM solves & FOM $\mathcal B_u/\mathcal B_u'$ &$n_\Q$ & $n_V$ &  o. iter \\ \midrule
		\FOM  & 1043  &35978 & 35934 & --\hspace{1.2mm} &  --\hspace{1.2mm} & 21\\
		\QFOM &  \phantom{1}338 & \phantom{1}8162 & \phantom{111}28 &    29 & --\hspace{1.2mm} &  27\\ 
		\QTRRB &  \phantom{11}91 & 86(+436) & \phantom{111}28 & 29  & 56 &  27\\ \bottomrule
	\end{tabular}
	\caption{Run~2: comparison of the performance of all algorithms. For the \QTRRB$\ $algorithm, the number of FOM solves needed for the error estimator is in brackets.}
	\label{fig: P1H1 Landweber problem table}
\end{table}

As in the first example, we observe in \cref{fig: P1H1 Landweber problem table} that the most expensive part of the \QTRRB$\ $ algorithm is the certification by the error estimator. The FOM gradients at the reconstructed parameter are of order $10^{-9}$. Further, the parameter-reduced methods obtain a low-dimensional representation of the optimal parameter by only using $29$ basis functions. The reconstructed parameters are depicted in \cref{fig: P1H1 Landweber problem pictures}.
The relative $L^2(\Omega)$ error norms are
\begin{align*}
	\frac{\|q^\mathrm{FOM}-q^{\Q_r}\|_{L^2(\Omega)}}{\|q^\mathrm{FOM}\|_{L^2(\Omega)}}\approx 2\,\% \quad\text{and}\quad \frac{\|q^\mathrm{FOM}-q^{\Q_r\text{-}V_r}\|_{L^2(\Omega)}}{\|q^\mathrm{FOM}\|_{L^2(\Omega)}}\approx2\,\%
\end{align*}
and the relative $H^1(\Omega)$ error norms are given as
\begin{align*}
	\frac{\|q^\mathrm{FOM}-q^{\Q_r}\|_{\Q}}{\|q^\mathrm{FOM}\|_{\Q}} \approx 20\,\%\quad\text{and}\quad \frac{\|q^\mathrm{FOM}-q^{\Q_r\text{-}V_r}\|_{\Q}}{\|q^\mathrm{FOM}\|_{\Q}}\approx 20\,\%.
\end{align*}
\ReA{The pointwise relative errors of the reduced methods w.r.t. to the FOM are depicted in \cref{fig: error diffusion 1 problem pictures}. The error is small in most regions and the maximum relative error peak for both methods is about $15\%$ which is quite good as $\Q\not\subset L^\infty(\Omega)$.}
\begin{figure}[ht]
	\hspace{0.8cm}\includegraphics[scale=1]{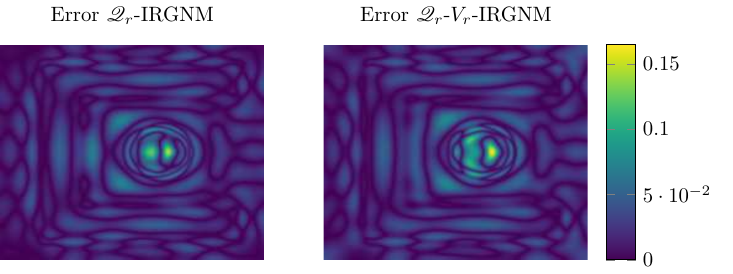}
    \caption[a]{Run~2: pointwise relative errors w.r.t. to the FOM reconstruction}
	\label{fig: error diffusion 1 problem pictures}
\end{figure}
%
\subsection{\ReA{Run~3: Reconstruction of the reaction coefficient -- Test~2}}\label{subsec: reconstr. reaction coeff2}
\ReA{
In this section, we consider another example for the reconstruction of the reaction coefficient in the setting of Section \ref{subsec. reconstr. reaction coefficient}. The exact parameter $q^\mathsf e$ contains differently regular parts: a Gaussian function, a hat function, and two discontinuous obstacles with local support on the background $q_0$ as shown in \cref{fig: error reaction 2 problem pictures}. Again, all methods capture the dominant characteristics of $q^\mathsf e$, and the reconstructions are depicted in \cref{fig: reaction 2 problem pictures}. The performance of the algorithms is depicted in \cref{fig: other reaction table} and in \cref{fig: computation time run 3 4} on the left. In this case, one observes a speedup of about $3$ to $4$ for the \QTRRB. The FOM gradient norms for all reconstructions were of the order of $10^{-9}$ and the $L^2(\Omega)$ errors are
\begin{align*}
	\frac{\|q^\mathrm{FOM}-q^{\Q_r}\|_\Q}{\|q^\mathrm{FOM}\|_\Q}\approx 10\,\% \quad\text{and}\quad \frac{\|q^\mathrm{FOM}-q^{\Q_r\text{-}V_r}\|_\Q}{\|q^\mathrm{FOM}\|_\Q}\approx 5\,\%.
\end{align*}
The pointwise relative errors w.r.t. the FOM are depicted in \cref{fig: error reaction 2 problem pictures}, which are in most regions significantly lower than $20\%$. In this case, the \QTRRB$\ $ admits a good approximation on the whole domain with a maximum relative error of $11\%$, but the \QFOM$\ $ struggles in approximating the boundary region, due to insufficient parameter snapshots, which can be explained as in Section \ref{subsec. reconstr. reaction coefficient}.
}
\begin{figure}[ht]
	\centering
	\includegraphics[scale=1]{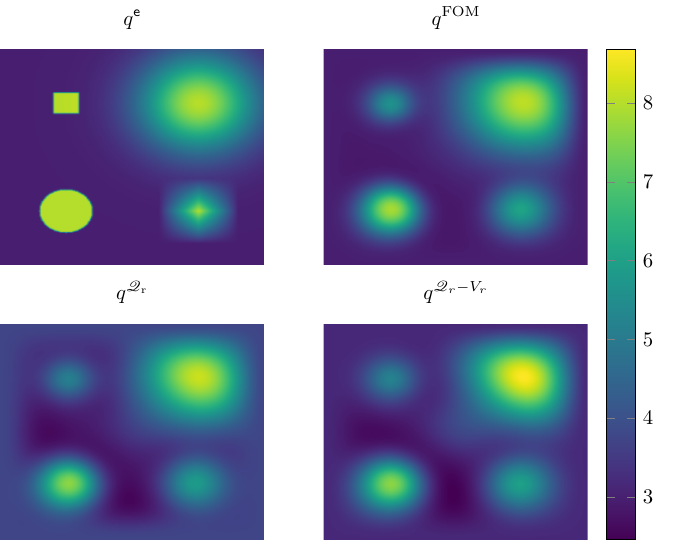}
    \caption[a]{Run~3: the exact parameter $q^\mathsf e$ and its three reconstructions $q^\mathrm{FOM}$, $q^\mathrm{\Q_r}$ and $q^{\Q_r\text{-}V_r}$}
	\label{fig: reaction 2 problem pictures}
\end{figure}
\begin{table}[ht] 
	\centering 
	\begin{tabular}{lcccccc}\toprule
		Algorithm & time [s]  & FOM solves & FOM $\mathcal B_u/\mathcal B_u'$ &$n_\Q$ & $n_V$ &  o. iter \\ \midrule
		\FOM  & 51  & \phantom{1}425 & 393 & -- &  -- & 15\\
		\QFOM  &  67 & 979 & \phantom{1}12 &    13 & --&  11\\ 
		\QTRRB &  14 & 24(+122) & \phantom{11}8 & \phantom{1}9 & 16 &  \phantom{1}7\\ \bottomrule
	\end{tabular}
	\caption{Run~3: Comparison of the performance of all algorithms. For the \QTRRB$\ $algorithm, the number of FOM solves needed for the error estimator is in brackets.}
	\label{fig: other reaction table}
\end{table}
\begin{figure}[ht]
	\hspace{0.8cm}\includegraphics[scale=1]{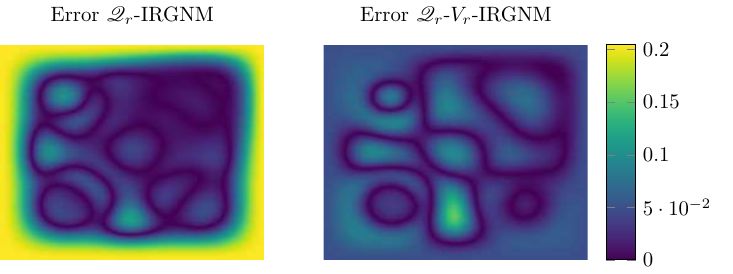}
    \caption[a]{Run~3: pointwise relative errors w.r.t. to the FOM reconstruction}
	\label{fig: error reaction 2 problem pictures}
\end{figure}
\begin{figure}[ht]
	\centering
	\includegraphics[scale=1]{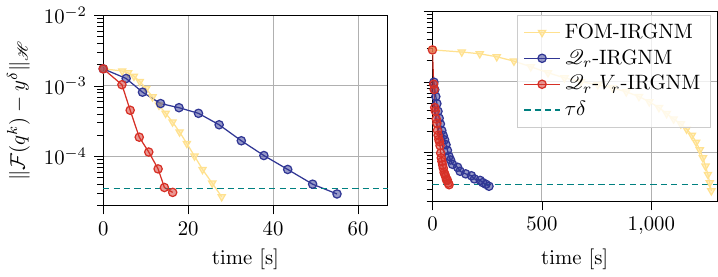}
    \caption[a]{The residuals $\|\F(q^k)-\yd\|_\H$ per computation time needed for iteration $k$ for Run~3 (left plot) and for Run~4 (right plot)} 
	\label{fig: computation time run 3 4}
\end{figure}
%
\subsection{\ReA{Run~4: Reconstruction of the diffusion coefficient -- Test~2}}\label{subsec: reconstr. diffusion coeff2}
\ReA{
We also consider another instance of the situation in Section \ref{subsec: reconstr. diffusion coeff}. The exact parameter is chosen as in the previous Section \ref{subsec: reconstr. reaction coeff2} and is depicted in \cref{fig: diffusion 2 problem pictures} with its three reconstructions that encapsulate key characteristics of the exact parameter. The performance of the algorithms is depicted in \cref{fig: other diffusion table} and in \cref{fig: computation time run 3 4} on the right, where we see a speedup of about $10$ ($3$) for the \QTRRB$\ $(\QFOM). The gradient norms for all reconstructions were of order $10^{-9}$ and the relative $L^2(\Omega)$ error norms are given as
\begin{align*}
	\frac{\|q^\mathrm{FOM}-q^{\Q_r}\|_{L^2(\Omega)}}{\|q^\mathrm{FOM}\|_{L^2(\Omega)}}\approx 2\,\% \quad\text{and}\quad \frac{\|q^\mathrm{FOM}-q^{\Q_r\text{-}V_r}\|_{L^2(\Omega)}}{\|q^\mathrm{FOM}\|_{L^2(\Omega)}}\approx2\,\%.
\end{align*}
The relative $H^1(\Omega)$ error norms are
\begin{align*}
	\frac{\|q^\mathrm{FOM}-q^{\Q_r}\|_{\Q}}{\|q^\mathrm{FOM}\|_{\Q}} \approx 18\,\%\quad\text{and}\quad \frac{\|q^\mathrm{FOM}-q^{\Q_r\text{-}V_r}\|_{\Q}}{\|q^\mathrm{FOM}\|_{\Q}}\approx 20\,\%.
\end{align*}
The pointwise relative errors of the reduced methods w.r.t. the FOM are depicted in \cref{fig: error diffusion 2 problem pictures} are lower than $12\%$.
}
\begin{figure}[ht]
	\centering
	\includegraphics[scale=1]{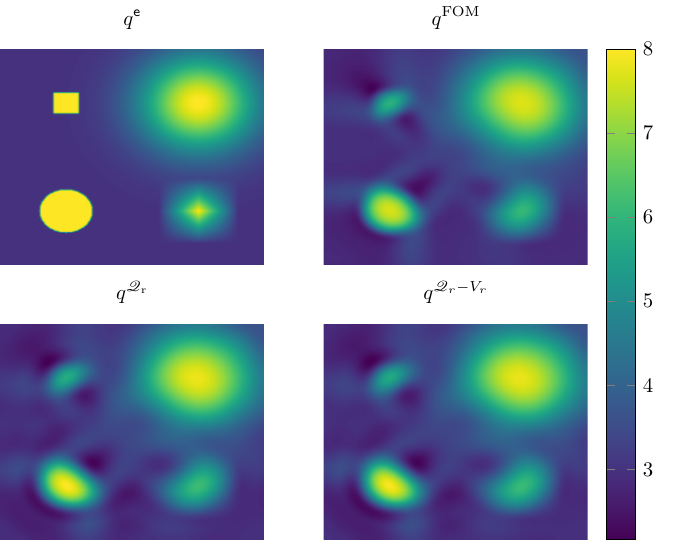}
    \caption[a]{Run~4: the exact parameter $q^\mathsf e$ and its three reconstructions $q^\mathrm{FOM}$, $q^\mathrm{\Q_r}$ and $q^{\Q_r\text{-}V_r}$}
	\label{fig: diffusion 2 problem pictures}
\end{figure}
\begin{table}[ht] 
	\centering 
	\begin{tabular}{lcccccc}\toprule
		Algorithm & time [s]  & FOM solves & FOM $\mathcal B_u/\mathcal B_u'$ &$n_\Q$ & $n_V$ &  o. iter \\ \midrule
		\FOM  &  1043 & 48723 & 48669 & -- &  -- & 26\\
		\QFOM  &  \phantom{1}333 & \phantom{1}6623 &\phantom{111}25 &    26 & --&  24\\ 
		\QTRRB &  \phantom{11}89 & 73(+314) & \phantom{111}24 & 25 & 48 & 23\\ \bottomrule
	\end{tabular}
	\caption{Run~4: Comparison of the performance of all algorithms. For the \QTRRB$\ $algorithm, the number of FOM solves needed for the error estimator is in brackets.}
	\label{fig: other diffusion table}
\end{table}
\begin{figure}[ht]
	\hspace{0.8cm}\includegraphics[scale=1]{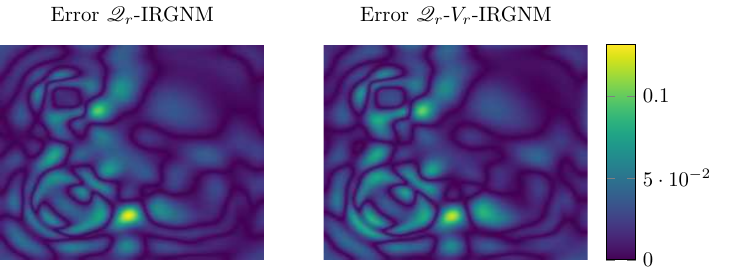}
    \caption[a]{Run~4: pointwise relative errors w.r.t. to the FOM reconstruction}
	\label{fig: error diffusion 2 problem pictures}
\end{figure}
%
\subsection{\ReA{Case study of the error estimator}}\label{subsec: error est case study}
%
\ReA{
In this subsection, we discuss condition \eqref{eq: error est assembly condition} used in the  \QTRRB by considering the situation in Section~\ref{subsec: reconstr. diffusion coeff} for an initial radius of $\delta^{(0)}=10$. We compare the following different strategies for evaluating the error estimator $\Delta_{\hat J}$, i.e., computing the residual norms in \eqref{eq: error estimator}:
\begin{enumerate}
    \item offline/online decomposition of the residual: every time the RB model is updated we compute the Riesz representatives corresponding to the new affine components of the bilinear form $a$,
    \item online evaluation of the residual norms: the residual norms are computed online by solving a linear system with complexity depending on the dimension of the FOM,
    \item mixed strategy using condition \eqref{eq: error est assembly condition}.
\end{enumerate}
The results are depicted in \cref{fig: error est evaluation figure} and \cref{fig: error est evaluation table}.
At the beginning of the iteration, both reduced bases are of small dimensions, which means the error estimator can be cheaply constructed in an offline/online manner. Additionally, the radius of the trust region is set adaptively during the first iterations, which means that the error estimator is often used during the Armijo backtracking (see the quantity $K_{\mbox{\tiny online}}^{k}$ in the first iterations on the right of \cref{fig: error est evaluation figure}). As the reduced bases grow due to unconditional enrichment, the cost of updating the error estimator progressively increases throughout the iteration (see the quantity $K_{\mbox{\tiny ass}}^{k}$ on the right in \cref{fig: error est evaluation figure}). Conversely, the ROM becomes increasingly accurate, leading to a diminished reliance on the error estimator. In the present example, condition \eqref{eq: error est assembly condition} switches at iteration $8$ from offline/online assembly to online computation of the error estimator and therefore prevents the case of expensively creating an error estimator that is not used at profitably many occasions. As a result, the mixed strategy follows the path of using minimal FOM solves for error estimation as can be seen in \cref{fig: error est evaluation figure} on the right and in \cref{fig: error est evaluation table}, which also is reflected in the computation times.
}
\begin{figure}[ht]
	\centering
    \includegraphics{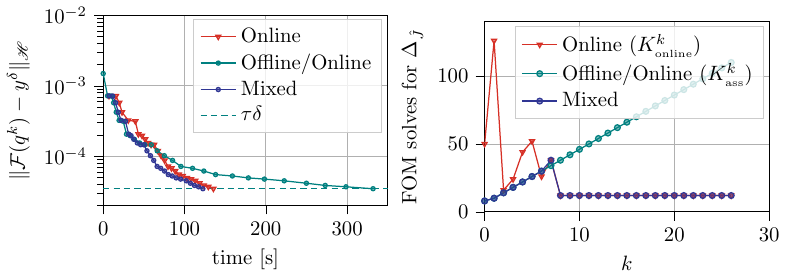}
    \caption[a]{Performance of the \QTRRB$\ $for different evaluation strategies of $\Delta_{\hat J}$. Left: discrepancy per computation time. Right: FOM solves for error estimation per outer iteration.\hfill\\}
	\label{fig: error est evaluation figure}
\end{figure}
\begin{table}[ht] 
	\centering 
	\begin{tabular}{lcc}\toprule
		Evaluation of $\Delta_{\hat{ J}}$ & time [s]  & FOM solves for $\Delta_{\hat{ J}}$  \\ \midrule
		Offline/online & 331  & 1682 \\
		Online  &  136 & 604 \\ 
		Mixed strategy using \eqref{eq: error est assembly condition} &  127 & 394 \\ \bottomrule
	\end{tabular}
	\caption{FOM solves for error estimation and computation time of the \QTRRB$\mbox{ }$ for different evaluation strategies of $\Delta_{\hat{ J}}$}
	\label{fig: error est evaluation table}
\end{table}
%
\section{Conclusion}
\label{Sec:5}

We introduced a new adaptive parameter and state reduction IRGNM for the solution of parameter-identification problems. The reduced parameter space is enriched with gradients, which are available anyway, if the reduced state space is enriched using both, the primal and dual solutions. The low dimensionality of the reduced parameter spaces allows us to efficiently build and certify a state-space RB model on the reduced parameter space. This enables to deal with high or even infinite-dimensional parameter spaces for an RB approximation. Numerical experiments with parameter spaces of the size of the FE discretization show the efficiency of the proposed approach for inverse parameter identification problems with distributed reaction or diffusion coefficients.

%
%
%
%
%
%
%
%
\section*{Declarations}
%
%
\begin{itemize}
\item {\bf Funding.} 
The authors acknowledge funding by the Deutsche Forschungsgemeinschaft for the project {\em Localized Reduced Basis Methods for PDE-Constrained Parameter Optimization}
under contracts OH 98/11-1; VO 1658/6-1. T. Keil and M. Ohlberger 
acknowledge funding by the Deutsche Forschungsgemeinschaft under Germany’s Excellence Strategy EXC 2044 390685587, Mathematics M\"unster: Dynamics -- Geometry -- Structure.
\item {\bf Competing Interests.}
The authors declare that they have no conflict of interest. 
\item {\bf Code availability.} The source code to reproduce the results of this article is available at \cite{source_code}.
\item {\bf Data availability.}
Data sharing is not applicable to this article as no datasets were generated or analyzed during the current study.
\item {\bf Authors' contributions.}
M.K. and T.K. implemented the code. M.K., T.K., M.O., and S.V. wrote the main manuscript text. All authors reviewed and approved the final manuscript.

\end{itemize}

\bibliography{biblio}

\end{document}